\documentclass[11pt,letterpaper]{amsart}
\usepackage[usenames,dvipsnames]{color}
\usepackage{amsthm,amsfonts,amssymb,amsmath,amsxtra}
\usepackage[all]{xy}
\SelectTips{cm}{}
\usepackage{xr-hyper}
\usepackage[pdfpagelabels,pdftex,hidelinks]{hyperref}
\hypersetup{
  colorlinks   = true, %Colours links
  urlcolor     = blue, %Colour for external hyperlinks
  linkcolor    = blue, %Colour of internal links
  citecolor   = blue, %Colour of citations
pdftitle={},
  pdfauthor={},
  pdfsubject={},
  pdfkeywords={},
  breaklinks=true,
  bookmarksopen=true,
  bookmarksnumbered=true,
  pdfpagemode=UseOutlines,
  plainpages=false}
\usepackage{cleveref}

\usepackage{verbatim}
\usepackage{caption}

\usepackage{lineno}
\usepackage{mathrsfs}

\RequirePackage{xspace}
\RequirePackage{etoolbox}
\RequirePackage{varwidth}
\RequirePackage{enumitem}
\RequirePackage{mathtools}
\RequirePackage{longtable}
\RequirePackage{multirow}

\def\ge{\geqslant}
\def\le{\leqslant}
\def\a{\alpha}
\def\b{\beta}
\def\g{\gamma}

\def\d{\delta}

\def\L{\Lambda}
\def\e{\epsilon}

\def\o{\omega}

\def\s{\sigma}
\def\t{\tau}
\def\th{\theta}
\def\k{\kappa}
\def\l{\lambda}
\def\z{\zeta}

\def\i{^{-1}}

\def\<{\langle}
\def\>{\rangle}

\newcommand{\fkX}{\ensuremath{\mathfrak{X}}\xspace}

\newcommand{\bG}{\mathbf G}

\newcommand{\BF}{\ensuremath{\mathbb {F}}\xspace}
\newcommand{{\BG}}{\ensuremath{\mathbb {G}}\xspace}
\newcommand{\BH}{\ensuremath{\mathbb {H}}\xspace}
\newcommand{\BI}{\ensuremath{\mathbb {I}}\xspace}

\newcommand{{\BK}}{\ensuremath{\mathbb {K}}\xspace}
\newcommand{\BL}{\ensuremath{\mathbb {L}}\xspace}

\newcommand{\BQ}{\ensuremath{\mathbb {Q}}\xspace}
\newcommand{\BR}{\ensuremath{\mathbb {R}}\xspace}

\newcommand{\BT}{\ensuremath{\mathbb {T}}\xspace}
\newcommand{\BU}{\ensuremath{\mathbb {U}}\xspace}

\newcommand{\BZ}{\ensuremath{\mathbb {Z}}\xspace}

\newcommand{\CA}{\ensuremath{\mathcal {A}}\xspace}
\newcommand{\CB}{\ensuremath{\mathcal {B}}\xspace}
\newcommand{\CC}{\ensuremath{\mathcal {C}}\xspace}

\newcommand{\CE}{\ensuremath{\mathcal {E}}\xspace}

\newcommand{\CG}{\ensuremath{\mathcal {G}}\xspace}
\newcommand{\CH}{\ensuremath{\mathcal {H}}\xspace}
\newcommand{\CI}{\ensuremath{\mathcal {I}}\xspace}

\newcommand{\CK}{\ensuremath{\mathcal {K}}\xspace}
\newcommand{\CL}{\ensuremath{\mathcal {L}}\xspace}

\newcommand{\CO}{\ensuremath{\mathcal {O}}\xspace}

\newcommand{\CQ}{\ensuremath{\mathcal {Q}}\xspace}
\newcommand{\CR}{\ensuremath{\mathcal {R}}\xspace}

\newcommand{\CT}{\ensuremath{\mathcal {T}}\xspace}

\newcommand{\Ad}{{\mathrm{Ad}}}
\newcommand{\ad}{{\mathrm{ad}}}

\DeclareMathOperator{\charac}{char}

\DeclareMathOperator{\Gal}{Gal}

\let\Im\relax
\DeclareMathOperator{\Im}{Im}

\DeclareMathOperator{\Nm}{Nm}

\newcommand{\red}{\ensuremath{\mathrm{red}}\xspace}

\DeclareMathOperator{\tr}{tr}

\newcommand{\ov}{\overline}

\def\bG{\mathbf{G}}

\def\brk{{\breve k}}
\def\dw{{\dot w}}

\def\COk{{\CO_{\brk}}}

\def\pr{{\rm pr}}
\def\tPhi{\widetilde \Phi}

\def\ind{{\rm ind}}
\def\Nm{{\rm Nm}}

\def\bx{{\mathbf x}}

\def\der{{\rm der}}
\def\ov{\overline}

\def\sc{{\rm sc}}

\def\unip{{\rm unip}}

\newtheorem{theorem}{Theorem}
\newtheorem{proposition}[theorem]{Proposition}
\newtheorem{lemma}[theorem]{Lemma}

\newtheorem{corollary}[theorem]{Corollary}

\theoremstyle{definition}

\newtheorem{remark}[theorem]{Remark}

\numberwithin{equation}{section}
\numberwithin{theorem}{section}

\setitemize[0]{leftmargin=*,itemsep=\the\smallskipamount}
\setenumerate[0]{leftmargin=*,itemsep=\the\smallskipamount}

\renewcommand{\to}{%
   \ifbool{@display}{\longrightarrow}{\rightarrow}%
   }
\let\shortmapsto\mapsto
\renewcommand{\mapsto}{%
   \ifbool{@display}{\longmapsto}{\shortmapsto}%
   }
\newlength{\olen}
\newlength{\ulen}
\newlength{\xlen}
\newcommand{\xra}[2][]{%
   \ifbool{@display}%
      {\settowidth{\olen}{$\overset{#2}{\longrightarrow}$}%
       \settowidth{\ulen}{$\underset{#1}{\longrightarrow}$}%
       \settowidth{\xlen}{$\xrightarrow[#1]{#2}$}%
       \ifdimgreater{\olen}{\xlen}%
          {\underset{#1}{\overset{#2}{\longrightarrow}}}%
          {\ifdimgreater{\ulen}{\xlen}%
             {\underset{#1}{\overset{#2}{\longrightarrow}}}
             {\xrightarrow[#1]{#2}}}}%
      {\xrightarrow[#1]{#2}}
   }
\makeatother
\newcommand{\xyra}[2][]{%
   \settowidth{\xlen}{$\xrightarrow[#1]{#2}$}%
   \ifbool{@display}%
      {\settowidth{\olen}{$\overset{#2}{\longrightarrow}$}%
       \settowidth{\ulen}{$\underset{#1}{\longrightarrow}$}%
       \ifdimgreater{\olen}{\xlen}%
          {\mathrel{\xymatrix@M=.12ex@C=3.2ex{\ar[r]^-{#2}_-{#1} &}}}%
          {\ifdimgreater{\ulen}{\xlen}%
             {\mathrel{\xymatrix@M=.12ex@C=3.2ex{\ar[r]^-{#2}_-{#1} &}}}
             {\mathrel{\xymatrix@M=.12ex@C=\the\xlen{\ar[r]^-{#2}_-{#1} &}}}}}%
      {\mathrel{\xymatrix@M=.12ex@C=\the\xlen{\ar[r]^-{#2}_-{#1} &}}}%
   }
\makeatletter
\newcommand{\xla}[2][]{%
   \ifbool{@display}%
      {\settowidth{\olen}{$\overset{#2}{\longleftarrow}$}%
       \settowidth{\ulen}{$\underset{#1}{\longleftarrow}$}%
       \settowidth{\xlen}{$\xleftarrow[#1]{#2}$}%
       \ifdimgreater{\olen}{\xlen}%
          {\underset{#1}{\overset{#2}{\longleftarrow}}}%
          {\ifdimgreater{\ulen}{\xlen}%
             {\underset{#1}{\overset{#2}{\longleftarrow}}}
             {\xleftarrow[#1]{#2}}}}%
      {\xleftarrow[#1]{#2}}
   }
\newcommand{\isoarrow}{%
   \ifbool{@display}{\overset{\sim}{\longrightarrow}}{\xrightarrow\sim}%
   }

\newcommand{\sm}{{\,\smallsetminus\,}}

\usepackage{upgreek}
\makeatletter
\newcommand{\colim@}[2]{%
  \vtop{\m@th\ialign{##\cr
    \hfil$#1\operator@font lim$\hfil\cr
    \noalign{\nointerlineskip\kern1.5\ex@}#2\cr
    \noalign{\nointerlineskip\kern-\ex@}\cr}}%
}
\newcommand{\colim}{%
  \mathop{\mathpalette\colim@{\rightarrowfill@\textstyle}}\nmlimits@
}
\newcommand{\prolim@}[2]{%
  \vtop{\m@th\ialign{##\cr
    \hfil$#1\operator@font lim$\hfil\cr
    \noalign{\nointerlineskip\kern1.5\ex@}#2\cr
    \noalign{\nointerlineskip\kern-\ex@}\cr}}%
}
\newcommand{\prolim}{%
  \mathop{\mathpalette\colim@{\leftarrowfill@\textstyle}}\nmlimits@
}
\makeatother

\DeclareMathOperator{\Perf}{Perf}

\begin{document}

\title{Deep level Deligne--Lusztig induction for tamely ramified tori}

\author{Alexander B. Ivanov}
\address{Fakult\"at f\"ur Mathematik, Ruhr-Universit\"at Bochum, D-44780 Bochum, Germany.}
\email{a.ivanov@rub.de}

\author{Sian Nie}
\address{Academy of Mathematics and Systems Science, Chinese Academy of Sciences, Beijing 100190, China}

\address{School of Mathematical Sciences, University of Chinese Academy of Sciences, Chinese Academy of Sciences, Beijing 100049, China}
\email{niesian@amss.ac.cn}

\begin{abstract} 
Deep level Deligne--Lusztig representations, which are natural analogues of classical Deligne--Lusztig representations, recently play an important role in geometrization of irreducible supercuspidals of $p$-adic groups. In this paper, we propose a construction of deep level Deligne--Lusztig varieties/representations in the tamely ramified case, extending previous constructions in the unramified case. As an application, under a mild assumption on the residue field, we show that each regular irreducible supercuspidal is the compact induction of a deep level Deligne--Lusztig representation, and generally, each irreducible supercuspidal is a direct summand of the compact induction of the cohomology of a deep level Deligne--Lusztig variety.

\end{abstract}

\maketitle

\section{Introduction}

Classical Deligne--Lusztig theory \cite{DeligneL_76} constructs varieties over the finite field $\BF_q$ equipped with the action of the finite group of Lie type $G(\BF_q)$. The cohomology of these varieties allows a uniform parametrization of all irreducible representations of such groups. Initiated by Lusztig \cite{Lusztig_79}, Deligne--Lusztig constructions associated to a reductive group $G$ over a local non-Archimedean field $k$ have been intensively studied in the literature over the past decade. We refer for example to the recent works \cite{ChenS_17, Chan_siDL, CI_loopGLn, ChenS_23, Chan24, ChanO_21, Nie_24, IvanovNie_25, CO_25} (and references therein), where the cohomology of deep level Deligne--Lusztig varieties (equipped with the action of a parahoric subgroup of $G(k)$) is investigated, and to \cite{Ivanov_DL_indrep} and \cite[\S9]{CI_loopGLn}, where $p$-adic Deligne--Lusztig spaces equipped with the action of the whole $p$-adic group are constructed.

Any (classical or deep level or $p$-adic) Deligne--Lusztig construction is attached to a rational maximal torus $T\subseteq G$. Due to its nature, in the case of a local field $k$ the construction so far only produced good output for \emph{unramified} tori $T$.
% (We mention that there were at least some attempts for ramified tori, see \cite{Stasinski_11,Ivanov_15_ADLV_GL2_ram,Ivanov_18_wild}.)

In this article we propose a new Deligne--Lusztig type construction for \emph{tamely ramified} maximal tori $T$ of $G$, and prove that it realizes the expected correspondence between smooth Howe factorizable  characters of $T(k)$ and smooth representations of $G(k)$. More precisely, we show, under mild assumptions on the residue field of $k$, that all the irreducible supercuspidals constructed by J.-K. Yu and Kaletha \cite{Yu_01,Kaletha_19} can be realized through weight spaces of $T(k)$ in the cohomology of these varieties. We refer to \cite{Stasinski_11,Ivanov_15_ADLV_GL2_ram,Ivanov_18_wild} for previous attempts to produce a reasonable Deligne--Lusztig construction for ramified tori. It would be interesting to compare our construction with these existing approaches.

\subsection{The construction and main result} \label{subsec:constr} We explain our results in more detail now. Let $k$ be as above and let $\BF_q$ be the residue field of $k$. Denote by $\brk$ the completion of a maximal unramified extension of $k$, by $\ov\BF_q$ the residue field of $\brk$, and by $F$ the Frobenius automorphism of $\brk$ over $k$. Let $G$ be a reductive group over $k$. We will identify $G$ with $G(\brk)$ so that $G(k) = G^F$. 

Let $s \in \BR_{\ge 0}$. Let $G_s := G_{\bx, 0} / G_{\bx, s+}$ be the $s$th Moy--Prasad quotient of the parahoric subgroup $G_{\bx, 0} \subseteq G$. We regard $G_s$ as a linear algebraic group over $\BF_q$. Using the positive loop functor, one can associated to any closed subgroup $M \subseteq G$ a closed subgroup $M_s$ of $G_s$, see \S\ref{sec:positive_loops}. 

Let $T \subseteq G$ be an elliptic maximal torus over $k$, which splits over a tamely ramified extension. Fix $r \in \BR_{\geq 0}$ and a smooth character $\phi \colon T(k) \to \overline\BQ_\ell^\times$ of depth $\leq r$. \emph{Suppose that $p>2$ and $\phi$ admits a Howe factorization $(G^i, \phi_i)_{i=-1}^d$ consisting of twisted Levi subgroups \[T = G^{-1} \subseteq G^0 \subsetneq G^1 \subsetneq \cdots \subsetneq G^d = G\] and characters $\phi_i: G^i(k) \to \ov\BQ_\ell^\times$ in the sense of \cite[\S3.6]{Kaletha_19}}. Motivated by Yu's construction of irreducible supercuspidals, we require that \[\tag{$\ast$} \text{$T$ is a maximally unramified elliptic maximal torus of $L := G^0$.}\] We will see that such pairs $(T, \phi)$ are sufficient for our purpose to realize all tame irreducible supercuspidals of $G(k)$.

Following \cite[\S 3]{Yu_01} one can associate to $(T, \phi)$ a Yu-type subgroup \[\CK \subseteq G_{\bx, 0}\] whose natural image in $G_r$ is denoted by $\CK_r$. We construct a variety \[Z_{\phi, r} \subseteq G_r\] equipped with a natural $\CK_r^F \times T_r^F$-action. Let us give a few more details on its construction. First, there is tower of normal subgroups \[G_{\bx ,r+} \subseteq \CE \subseteq \CK^+ \subseteq \CH \subseteq \CK = \CH L_{\bx, 0},\] and the short exact sequence \[1 \to \CK^+ / \CE \to \CH / \CE \to \CH / \CK^+ \to 1\] gives a central extension of the symplectic-type space $\CH / \CK^+$. On the other hand, by ($\ast$) there exists a Borel subgroup $B = T U \subseteq L$ over $\brk$ with unipotent radical $U$. We show that there exists a subgroup \[\CE \subseteq \CL \subseteq \CH\] such that $\CL / \CE$ is a section of a Lagrangian space in $\CH / \CK^+$ both normalized by $B_r$. Then we put $\CI_r = U_r \CL_r = \CL_r U_r$ and define \[Z_{\phi, r} = \{g \in \CK_r; g\i F(g) \in F \CI_r\},\] which admits a natural action by $\CK_r^F \times  T_r^F$.
\begin{remark}
   The construction of $Z_{\phi, r}$ extends that in \cite{Nie_24} for the unramified case (see also \cite{ChenS_17, ChenS_23} for the generic case). A similar construction is first given in \cite{CKZ}. The idea of using Lagrangian subspace was communicated to the authors by Xinwen Zhu. 
\end{remark}

In practice, we put $\BK = \CK / \CE$, and write $\BT,\BL$ and $Z$ for the natural images of $T_r,L_r$ and $Z_{\phi, r}$ in $\BK$ respectively. The cohomology groups $H^i_c(Z,\ov\BQ_\ell)$ inherit natural actions by $\BK^F \times \BT^F$, and we denote by $H^i_c(Z,\ov\BQ_\ell)[\phi]$ their subspaces on which $\BT^F$ acts via $\phi$. Let
\[
\CR^{\BK}_{\BT}(\phi) = H^*_c(Z, \ov\BQ_\ell)[\phi] := \sum_i (-1)^i H^i_c(Z,\ov\BQ_\ell)[\phi],
\]
which is a virtual $\BK^F$-module, or a virtual $\CK^F$-module by inflation. Let $W_{L_r}(T_r) = \{g \in L_r \colon gT_rg^{-1} = T_r\}/T_r$. Note that $W_{L_r}(T_r) = W_{\BL}(\BT)$ is the Weyl group of $\BT$ in $\BL$. We will always denote this group by $W_{\BT}(\BL)$.

Our first main result is the following Mackey type formula of $\CR^{\BK}_{\BT}(\phi)$.
\begin{theorem} \label{main-1}
    We have \[\<\CR^{\BK}_{\BT}(\phi), \CR^{\BK}_{\BT}(\phi)\>_{\BK^F} = \sharp \{w \in W_{\BL}(\BT)^F; {}^w \phi|_{T_r^F} = \phi |_{T_r^F}\}.\]
    %where $ W_{L_r}(T_r)^F$ is the Weyl group of $T_r$ in $L_r$.
\end{theorem}

To compute $\CR^{\BK}_{\BT}(\phi)$, we introduce the following $\BK^F$-module of Yu-type
\[
R^{\BK}_{\BT}(\phi) = \k \otimes R_{T_0}^{L_0}(\phi_{-1}),
\]
where \begin{itemize}
    \item $\k$ is the Weil--Heisenberg representation as in \cite[\S2.5]{Fin_21b} (see also \cite[\S4 and \S11]{Yu_01}) attached to the positive-depth part $(G^i, \phi_i)_{i=0}^d$ of the Howe factorization of $\phi$;

    \item $R_{T_0}^{L_0}(\phi_{-1})$ is the classical Deligne--Lusztig representation of $L_0^F$ attached to the depth-zero part $\phi_{-1}$ of $\phi$.
\end{itemize}

Our second result is the following.
\begin{theorem} \label{main-2}
Suppose $q$ is sufficiently large. Then \[ \CR^{\BK}_{\BT}(\phi) \cong (-1)^{d(\phi)} R^{\BK}_{\BT}(\phi^\dag),\] where $d(\phi)$ is an integer depending on $\phi$, $\phi^\dag = \phi \cdot \varepsilon_\phi$ and $\varepsilon_\phi = \prod_{i=1}^d \e_{\sharp, \bx}^{G^i/G^{i-1}}$ is a product of depth zero quadratic characters $\e_{\sharp, \bx}^{G^i/G^{i-1}}$ of the maximal bounded subgroup of $T^F$ as in \cite[Definition 3.1]{FKS}.
\end{theorem}

\begin{remark}
    In the unramified case, \Cref{main-1} is proved in \cite[Proposition 5.5]{Nie_24}, and \Cref{main-2} is proved in \cite{LN_25}.
\end{remark}

\subsection{The application} 
Now we discuss applications on irreducible supercuspidals of $G(k)$.

Let $(T, \phi)$ be a tame elliptic regular pair in the sense of \cite[Definition 3.7.5]{Kaletha_19}. In particular, our assumption ($\ast$) is satisfied. Following \cite[\S 3.7]{Kaletha_19}, one can associate to $(T, \phi)$ an irreducible supercuspidal $\pi_{(T, \phi)}$ of $G(k)$, which is referred to as a regular irreducible supercuspidal. 

On the other hand, by generalizing the geometric construction of $\pi_{(T, \phi)}$ in the depth zero case (see \cite[\S3.4.3]{Kaletha_19}), we can extend the $\CK^F$-modules $H_c^i(Z, \ov\BQ_\ell)[\phi]$ to $\CK^F T^F$-modules. Thus the alternating sum \[\hat\CR_{\BT}^{\BK}(\phi) := H_c^*(Z, \ov\BQ_\ell)[\phi]\] is a virtual $\CK^F T^F$-module. 

Our first application is the following cohomoligical realization of regular irreducible supercuspidals.
\begin{theorem} \label{app-1}
    Assume $q$ is sufficiently large and $p$ does not divide the order of the Weyl group of $G$. Let $(T, \phi)$ be a tame elliptic regular pair. Then \[\pi_{(T, \phi)} \cong (-1)^{d(\phi) + r_G -r_T} \text{c-}\ind_{\CK^F T^F}^{G^F} \hat \CR_{\BT}^{\BK}(\phi^\dag),\] where $r_G$ and $r_T$ are the splitting ranks of $G$ and $T$ respectively.
\end{theorem}

For general irreducible supercuspidals we have the following exhaustion result.
\begin{proposition} \label{app-2}
    Assume $q$ is sufficiently large. For each tame irreducible supercuspidal $\pi$ of $G(k)$, there exits some pair $(T, \phi)$ as in \S\ref{subsec:constr} and $i \in \BZ_{\ge 0}$ such that $\pi$ is direct summand of the compact induction \[\text{c-}\ind_{Z(G)^F \CK^F}^{G^F} H_c^i(Z, \ov\BQ_\ell)[\phi].\] Here we extend the $\CK^F$-module $H_c^i(Z, \ov\BQ_\ell)[\phi]$ to a $Z(G)^F \CK^F$-module on which $Z(G)^F$ acts via $\phi$.
\end{proposition}
\begin{remark}
    In the unramified case, \Cref{app-1} and \Cref{app-2} are proved in \cite{CO_25} and \cite{Nie_24, LN_25} independently.
\end{remark}

Note that by exhaustion results of Kim and Fintzen \cite{Kim, Fin_21a}, any irreducible supercuspidal $G(k)$-representation is tame when $p$ does not divide the order of the absolute Weyl group of $G$.

\subsection{The strategy}
Now we discuss the ideas for the proofs of the main results, which combine the methods/strategies from \cite{Nie_24}, \cite{IvanovNie_24} and \cite{CO_25}. Let us explain in more details. 

To establish the Mackey type formula in \Cref{main-2}, we follow the proof of \cite[Proposition 5.5]{Nie_24} by reducing the problem to the computation of the cohomology group $H_c^*(X^H, \ov\BQ_\ell)$, where $H \subseteq T_r \times T_r$ is a torus and $X$ is a certain $H$-invariant subset of the intertwining variety. In the unramified case, the fixed-point set $X^H$ is a finite set and the computation follows trivially. While in the tamely ramified case, $X^H$ is no longer finite in general. To overcome this difficulty, we provide a detailed study of $X^H$ and manage to determine its cohomology based on an idea from \cite{IvanovNie_24}.

To prove \Cref{main-2}, we follow the strategy of \cite{CO_25}. It consists of two steps. First we show that $\CR^{\BK}_{\BT}(\phi)$ and $(-1)^{d(\phi)} R^{\BK}_{\BT}(\phi^\dag)$ have the same expansion in terms of their associated Green functions respectively. This step follows in a similar way as in \cite{CO_25}. The remaining task is to show the compatibility between these two type of Green functions. As observed in \cite{CO_25}, by using \Cref{main-1} and the largeness assumption on $q$, it suffices to show that
\[\tag{$\flat$} \text{ $\CR^{\BK}_{\BT}(\phi) = \pm \k \otimes R_{T_0}^{L_0}(\phi_{-1}^\dag)$ if $\pm \CR^{\BK}_{\BT}(\phi)$ is irreducible.}\] In the unramified case, the proof of ($\flat$) is  based on comparing the values of sufficiently many regular elements for $G_r$ on both sides. However, in our case, the set of regular elements is no longer large enough in general. To handle this difficulty, we first note that $\CR^{\BK}_{\BT}(\phi) = \pm \k \otimes \rho$ for some irreducible $L_0^F$-module $\rho$. So it suffices to show $\rho \cong \pm R_{T_0}^{L_0}(\phi_{-1}^\dag)$. To this end, we establish a trace formula of $\CR^{\BK}_{\BT}(\phi)$ on the set $\BK_{ss}^F$ of semi-simple elements of $\BK^F$. This is built on a concentration theorem of the cohomology of the positive-depth part of $Z$ (see \Cref{concentration}), which extends \cite[Proposition 6.2]{Nie_24} in the unramified case. Then by comparing it with the trace formula of $\k \otimes \rho$, we deduce that $\rho$ and $\pm R_{T_0}^{L_0}(\phi_{-1}^\dag)$ coincide on $\BK_{ss}^F$. As $q$ is large enough, $\BK_{ss}^F$ has sufficiently many regular elements for $L_0$. Thanks to \cite[Theorem 1.2]{CO_25}, the two irreducible characters $\rho $ and $\pm R_{T_0}^{L_0}(\phi_{-1}^\dag)$ coincide with each other as desired.

\subsection{The outline}
The paper is organized as follows. In \S\ref{sec:positive_loops} we lay down the basic set-up, and introduce subgroups attached to a Howe factorization in Yu's framework. In \S\ref{sec:Heisenberg} we provide the construction of deep level Deligne--Lusztig varieties/representations in tamely ramified case. In \S\ref{sec:Mackey} we establish the Mackey formula for the inner product of deep level Deligne--Lusztig representations. In \S\ref{sec:concentration}, we extend a concentration theorem from \cite{Nie_24} on cohomology groups to the tamely ramified case, which play an essential role in the proof of our main results. In \S\ref{sec:character} we  prove a character formula for $\CR^{\BK}_{\BT}(\phi)$ by following the approach of \cite{DeligneL_76} and \cite{CO_25}. In \S\ref{sec:R-mod} we introduce and study another representation $\CR^{\BK}_{\BT}(\phi^\dag)$ as a counterpart of $\CR^{\BK}_{\BT}(\phi)$. In \S\ref{sec:Green} we prove \Cref{main-2} and \Cref{app-1} by following a strategy of \cite{CO_25} on Green functions of $\CR^{\BK}_{\BT}(\phi)$ and $R^{\BK}_{\BT}(\phi^\dag)$. In the last section, we proved the exhaustion result \Cref{app-2} on irreducible supercuspidals.

\subsection*{Acknowledgments.} We are grateful to Xinwen Zhu for sharing the idea of Lagrangian-based  construction. We are also appreciative of Tasho Kaletha for his kindly answering several technical questions. The first named author gratefully acknowledges the support of the German Research Foundation (DFG) via the Heisenberg program (grant nr. 462505253).

\section{Preliminaries}
We denote by $k$ a local non-Archimedean field with residue field $\BF_q$ of characteristic $p$. We write $\brk$ for the completion of a maximal unramified extension of $k$, and $F$ for the Frobenius automorphism of $\brk$ over $k$. We denote by $\CO_k$ and $\COk$ the integer rings of $k$ and $\brk$ respectively. Let $\varpi$ be a fixed uniformizer of $k$.

\subsection{Positive loops}\label{sec:positive_loops}
Let $E/\breve k$ be a finite tamely ramified extension. It is automatically Galois with cyclic Galois group $\Gal(E/\breve k) = \langle \tau \rangle$. Let $\CO_E$ denote the integers of $E$. Note that $E$ and $\breve k$ have the same residue field $\ov\BF_q$. Let $\Perf_{\ov\BF_q}$ denote the category of perfect $\ov\BF_q$-schemes. Let $\fkX$ be an $\CO_E$-scheme. The positive loop functor (also called jet scheme) of $\fkX$ is defined by
\[
L_E^+ \fkX \colon \Perf_{\ov\BF_q}^{\rm op} \to {\rm Sets}, \quad R \mapsto \fkX(W_{\CO_E}(R)),
\]
where $W_{\CO_E}(R) = \CO_E \otimes_{\ov\BF_q} R$ if ${\rm char}\, E>0$ and $W_{\CO_E}(R) = W(R) \otimes_{\breve \BZ_p} \CO_E$ otherwise (here, $W(R)$ denote the $p$-typical Witt-vectors of $R$). We only consider schemes $\fkX$ which are affine and of finite type over $\CO_E$, in which case $L_E^+\fkX$ is representable by a affine scheme perfectly of finite presentation over $\ov\BF_q$.

When $E = \breve k$, we also write $L^+\fkX$ for $L_E^+\fkX$.

\subsection{Moy--Prasad quotients}\label{sec:setup}
Let $G$ be a reductive group over $k$, which splits over a tamely ramified extension of $k$.
% We assume that $p \neq 2$ is not a bad prime for $G$ and $p \nmid |G_\der|$.
Let $\bx$ be a point in the Bruhat-Tits building of $G$ over $k$. We denote by $\CG_{\bx, 0}$ the associated parahoric $\CO_k$-group model of $G$. For $0 \le r \in \widetilde \BR := \BR \sqcup \{s+; s \in \BR\}$ let $\CG_{\bx, r}$ be the $r$th Moy--Prasad subgroup of $\CG_{\bx, 0}$.
For any $0 \le s \le r \in \widetilde \BR$, $L^+(\CG_{\bx,r+} \times_{\CO_k} \CO_{\breve k}) \subseteq L^+(\CG_{\bx,s} \times_{\CO_k} \CO_{\breve k})$ is a subgroup. We put
\[G_{s:r} = L^+(\CG_{\bx,s} \times_{\CO_k} \CO_{\breve k}) / L^+(\CG_{\bx,r+} \times_{\CO_k} \CO_{\breve k}),\]
which is a pfp perfectly smooth affine group scheme over $\ov\BF_q$.
Note that as $F$ fixes $\bx$ it acts naturally on $G_{s:r}$, equipping it with an $\BF_q$-rational structure. Moreover, as $L^+\CG_{\bx,r+}$ is connected, $H^1(\breve k / k,L^+\CG_{\bx,r+}) = 0$ by Lang's theorem, and hence $G_{s:r}^F = \CG_{\bx,s}(\CO_k)/\CG_{\bx,r+}(\CO_k)$.

More generally, for any $E/\breve k$ as in \S\ref{sec:positive_loops} and $0 \le s \le r \in \widetilde \BR$, we also may form the pfp perfectly smooth affine group scheme
\[ {}^E G_{s:r} = L^+_E(\CG_{\bx,s} \times_{\CO_k} \CO_E) / L^+_E(\CG_{\bx,r+} \times_{\CO_k} \CO_E),\]
over $\ov\BF_q$. It is equipped with a natural action of  $\Gal(E/\breve k) = \langle \tau \rangle$. The (pro-version of) Lang--Steinberg theorem does not apply as the $\tau$-fixed points are not profinite, but we still have the following.

\begin{lemma}[see page 32 of \cite{PappasR_08}] \label{lm:galois_coho_vanishes}
We have $H^1(\Gal(E/\breve k), L^+_E\CG_{\bx,r+}) = 0$. In particular, $({}^E G_{s:r})^\tau = G_{s:r}$ (as perfect $\ov\BF_q$-group schemes).
\end{lemma}
\begin{proof}
As $r+ > 0$, the group $L^+_E\CG_{\bx,r+}(\ov\BF_q) = \CG_{\bx,r+}(\CO_E)$ is an inverse limit of groups which are iterated extensions of $\CO_E$-modules. Since the order of $\tau$ is coprime to $p$, the cohomology of each filtration step vanishes, and hence $H^1(\Gal(E/\breve k), L^+_E\CG_{\bx,r+}) = 0$ by d\'evissage. The last claim of the lemma follows from the exact sequence of non-abelian group cohomology and the fact that a perfectly smooth affine group scheme is (as a variety) uniquely determined by its $\BF_q$-points.
\end{proof}

Let $H \subseteq G$ be a closed subgroup. Following \cite[\S2.5]{CI_MPDL} one can construct an $\ov\BF_q$-rational closed subgroup $H_{s:r} \subseteq G_{s:r}$ in a similar way. We put $H_r = H_{0:r}$ for simplicity. If, moreover, $H$ is a $k$-rational subgroup, then $H_{s:r}$ is defined over $\BF_q$, and we still denote by $F$ the induced Frobenius automorphisms on $H$ or $H_{s:r}$. Similarly, we have the groups ${}^E H_{s:r}$ and ${}^E H_r$. If $H$ is connected, then $({}^E H_{s:r})^\tau = H_{s:r}$.

\subsection{Tamely ramified tori}\label{sec:tame_tori}
Let $T$ be a $k$-rational maximal torus of $G$ which splits over a finite tamely ramified extension $K/k$. Let $k_f / k$ be the maximal unramified extension in $K$. We write $[K: k] = ef$, with $e = [K:k_f]$ the (prime to $p$) ramification index and $f = [k_f:k]$ the inertia degree of $K/k$. There exist $a \in k_f$ and a uniformizer $\pi$ of $K$ such that $\pi^e = a\varpi$. We write $E = K\breve k$ for the maximal unramified extension of $K$. Then $E/\breve k$ is Galois with cyclic Galois group of order $e$ and we denote by $\tau$ one of its generators. Then $\tau(\pi) = \zeta \pi$ for a primitive $e$th root of unity $\zeta \in \breve k$. We write $E^\times = \CO_E^\times$, and for $r\geq 0$ we write $E^\times_r = 1+ \pi^{\lceil er \rceil}\CO_E$.

We write $W = N_G(T)/T$ for the Weyl group of $T$, $\Phi = \Phi(G_E, T_E)$ for the set of roots of $T_E$ in $G_E$, ${}^EG^\alpha \subseteq G_E$ for the root subgroup of $\alpha$, and $\tilde\Phi = \Phi \times \frac{1}{e}\BZ \,\cup\, \frac{1}{e}\BZ_{\geq 0}$ for the set of affine roots. If $f \in \Phi \times \frac{1}{e}\BZ$, we write $\alpha_f \in \Phi$, $n_f \in \frac{1}{e}\BZ$ such that $f = (\alpha_f,n_f)$. The affine root subgroup of $f$ is ${}^EG^f = ({}^E G^\alpha)_ {n_f:n_f}$. If $\alpha \in \Phi$, we write $\alpha^\vee \colon \bG_{{\rm m},K} \to T_K$ for the corresponding coroot and for $x \in E^\times$, $\alpha^\vee \otimes x$ for the image of $x$ under $\alpha^\vee$. If $x \in E_r^\times$ for some $r\in \widetilde \BR_{\geq 0}$, we also write $\alpha^\vee \otimes x$ for the image of $x$ in ${}^ET_{r:r'}$ (for some $r' \geq r$).

We fix a Chevalley system for $(G_K,T_K)$. Recall from \cite[\S4.2]{BruhatT_84} or \cite[\S1.2]{Adler_98} that this includes a set of $K$-rational isomorphisms $u_\alpha \colon \bG_{{\rm a},K} \to G^{\alpha}$ for $\alpha \in \Phi = \Phi(G_K,T_K)$ satisfying several properties. For an affine root $f =(\alpha,n_f)$ with $\alpha \in \Phi$, we have the parametrization 
\begin{equation} \label{eq:parametrization_affine}
\bG_{{\rm a},\ov\BF_q} \stackrel{\sim}{\to} {}^EG^f, \quad \text{induced by } \quad x \mapsto u_\alpha(\pi^{en_f}[x]), 
\end{equation}
where $[x] = x$ if $\charac k> 0$, and $[x]$ is the Teichm\"uller lift otherwise.

Note that $\tau$ acts on $\Phi$ and on $\tilde\Phi$. We have the formula
\[
\tau(u_\alpha(x)) = u_{\tau(\alpha)}(a_\alpha\tau(x)) \quad \forall x \in E
\]
for some constant $a_\alpha \in K^\times$. Write $a_\alpha = \pi^{v_E(a_\alpha)} \tilde a_\alpha$ with $\tilde a_\alpha \in \CO_K^\times$. 

\begin{lemma}\label{lm:relation_between_constants}
Let $f=(\alpha,n_f)$ be an affine root with $\alpha \in \Phi$, so that $\tau(f) = (\tau(\alpha),n_{\tau(f)})$. 
\begin{itemize}
\item[(1)] We have 
\begin{equation}\label{eq:action_tau_affine_root_subgroups}
\tau(u_f(y)) = u_{\tau(f)}(\zeta^{e n_f}\bar a_\alpha y) \quad \forall y \in \ov\BF_q 
\end{equation}
where $\bar a_\alpha  = \tilde a_\alpha \mod \pi\CO_K \in \BF_{q^f}^\times = (\CO_K/\pi\CO_K)^\times$.
\item[(2)] We have $a_{-\alpha} = a_\alpha^{-1}$ and $\bar a_{-\alpha} = \bar a_\alpha^{-1}$.
\item[(3)] Let $f \in \tilde\Phi$ and $s=f(\bx)$. Supppose that for some $N>0$, $\tau^{2N}(f)=f$ and $\tau^{2N}$ acts trivially on ${}^EG^f$. Then $\zeta^{2eNs}=1$.
\end{itemize}

\end{lemma}
\begin{proof}
(1) follows directly from \eqref{eq:parametrization_affine} and the above formula for the action of $\tau$ on $u_\alpha(x)$. For (2) note that $u_\alpha(x)u_{-\alpha}(y)u_\alpha(x)$ lies in $N_G(T)(E)$ if and only if $xy=-1$. Applying $\tau$ to this expression and noticing that $\tau$ preserves $N_G(T)(E)$, one deduces $a_\alpha \cdot a_{-\alpha}= 1$. The second claim follows from the first by (1). To prove (3), first note that as $\bx$ is $\tau$-invariant, we have 
\begin{equation}\label{eq:2s_equal_fplustNf} 
2s = \tau^N(f)(\bx) + f(\bx) = \alpha(\bx) + n_{\tau^N(f)} + (-\alpha)(\bx) + n_f = n_f + n_{\tau^N(f)},
\end{equation}
as $\alpha$ is a linear function. Then we compute
\begin{align*}
u_f(y) &= \tau^{2N}(u_f(y)) \\ 
&= \tau^N\left(u_{\tau^N(f)}\left(\zeta^{e\sum_{i=0}^{N-1} n_{\tau^i(f)}} \cdot \prod_{i=0}^{N-1} \bar a_{\tau^i(\alpha)} \cdot y\right)\right) \\
&= u_f\left(\zeta^{e\sum_{i=0}^{N-1} n_{\tau^i(f)} + n_{\tau^{i+N}(f)}} \cdot \prod_{i=0}^{N-1} (\bar a_{\tau^i(\alpha)} \bar a_{\tau^{i+N}(\alpha)})\cdot y \right) \\
&= u_f\left(\zeta^{e\sum_{i=0}^{N-1} n_{\tau^i(f)} + n_{\tau^{i+N}(f)}} \cdot  y \right) \\
&= u_f\left(\zeta^{2seN} \cdot  y \right)
\end{align*}
where the first equation is by assumption, the second and the third are by part (1) of the lemma, the fourth is by part (2) of the lemma, the fifth is by \eqref{eq:2s_equal_fplustNf}. As this holds for any $y \in \ov\BF_q$, the claim follows. 
\end{proof}

\subsection{Subgroups attached to a torus character}\label{sec:Howe}
We keep the setup of \S\ref{sec:tame_tori} and assume additionally that $T$ is elliptic in $G$. \emph{We assume that $p>2$}. We fix an $r \in \BR_{\geq 0}$, a smooth character
\[ \phi \colon T^F = T(k) \to \overline\BQ_\ell^\times \]
of depth $\leq r$, and we denote the character induced by $\phi$ on $T_{0:r}^F$ by the same letter.
% We also denote by $\phi$ its restriction to $T_{0:r}^F$.
Moreover, we assume that
\begin{equation}\label{eq:assumption_Howe}
\text{$\phi$ admits a Howe factorization $(G_i,r_i,\phi_i)_{i=-1}^d$ as in \cite[\S3.6]{Kaletha_19}}.
\end{equation}
Recall that this means that \[\phi = \phi_{-1} \prod_{i=0}^d \phi_i |_{T(k)},\] where $\phi_{-1}$ is a character of $T(k)$ of depth zero, and $\L := (G_i,r_i,\phi_i)_{i=0}^d$ is {\it generic datum} defined by the following conditions:
% \begin{equation}\label{eq:assumption_p}
% \text{$p$ is not a bad prime for $G$ and $p$ does not divide $|G_\der|$.}
% \end{equation}
% (same assumption as in \cite[\S3.6]{Kaletha_19}).
\begin{itemize}
\item $T= G^{-1} \subseteq G^0 \subsetneq G^1 \subsetneq \cdots \subsetneq G^d = G$ are $k$-rational Levi subgroups of $G$;
\item $0 =: r_{-1} < r_0 < \cdots < r_{d-1} \le r_d \le r$ if $d \ge 1$ and $0 \le r_0$ if $d = 0$;
\item $\phi_i: G^i(k) \to \ov\BQ_\ell^\times$ is a character of depth $r_i$, and trivial on $G_\der^i(k)$ \footnote{In \cite[Definition 3.6.2]{Kaletha_19}, $\phi_i$ is only required to be trivial over $G^i_\sc(k)$. However, the proof of \cite[Lemma 3.6.9]{Kaletha_19} shows that $\phi_i$ can be chosen to be trivial over $G_\der^i(k)$.} for $-1 \le i \le d$;
\item $\phi_i$ is of depth $r_i$ and is $(G^i, G^{i+1})$-generic in the sense of \cite[\S 9]{Yu_01} for $0 \le i \le d-1$.
\end{itemize}

\begin{remark}\label{rem:when_Howe_exists}
If $p$ is neither a bad prime for $G$, nor divides $|\pi_1(G_\der)|$, then any character $\phi$ admits a Howe factorization by \cite[Proposition 3.6.7]{Kaletha_19}.
\end{remark}

The role of the Levi subgroup $G^0$ will be somewhat special, and we will often denote \[ L = G^0.\] Moreover, for $0 \leq i \leq d$, we set \[ s_i = \frac{r_i}{2}.\] Write $T^i_\der = G^i_\der \cap T$. Following \cite{Yu_01} (see also \cite[\S3.2]{Nie_24} and \cite[\S4]{IN_sheaves}), we associate to the generic datum $\L$ the following subgroups of $G$.
% Let $\ov U$ be the opposite of $U$ and set $T^i_\der = G^i_\der \cap T$. We consider the following subgroups.
\begin{align*}
    \widetilde\CK_\L &= G^0_{[\bx]} G^1_{\bx, s_0} \cdots G^d_{\bx, s_{d-1}} \\
    \CK_\L &= G^0_{\bx, 0} G^1_{\bx, s_0} \cdots G^d_{\bx, s_{d-1}} \\
    \CH_\L &= G^0_{\bx, 0+} G^1_{\bx, s_0} \cdots G^d_{\bx, s_{d-1}} \\
    \CK_\L^+ &= G^0_{\bx, 0+} G^1_{s_0+} \cdots G^d_{\bx, s_{d-1}+} \\
    \CT_\L & = (T^0_\der)_{\bx, 0+} (T^1_\der)_{\bx, r_0+} \cdots (T^d_\der)_{\bx, r_{d-1}+}.\\
\end{align*} Here $[\bx]$ is the natural image of $\bx$ in the reduced Bruhat-Tits building of $G^0$ and  $G^0_{[\bx]}$ denotes the stabilizer of $[\bx]$ in $G^0$.

We define another subgroup $\CE_\L$ as follows. First we put \begin{align*}{}^E \CE_\L =({}^E G^0_\der)_{\bx; 0+,0+} ({}^E G^1_\der)_{\bx; r_0+, s_0+} \cdots ({}^E G^d_\der)_{\bx; r_{d-1}+, s_{d-1}+} {}^E G_{\bx, r+}, \end{align*} where $({}^E G_\der^i)_{\bx; r_{i-1}, s_{i-1}}$ is the subgroup generated by $({}^E G^i_\der)_{r_{i-1}+}$ and ${}^E G^f$ for $f \in \tPhi_{G^i} \setminus \tPhi_{G^{i-1}}$ such that $f(\bx) > s_{i-1}$. Here $\tPhi_{G^i}$ denotes the root system of $T_E$ in $G_E^i$ for $0 \le i \le d$. Note that ${}^E \CE_\L$ is $\tau$-stable and we set
\begin{align*}
\CE_\L &= ({}^E \CE_\L)^\tau,
\end{align*}
which is a subgroup of $\CK_\L^+$. 

We put $\widetilde \CK_{\L, r} = \widetilde\CK_\L / G_{\bx, r+}$ and ${}^E \widetilde \CK_{\L, r} = {}^E \widetilde\CK_\L / {}^E G_{\bx, r+}$. Other subgroups $\CK_{\L, r}$, ${}^E \CK_{\L, r}$ and so on are defined in the same way. Notice that $\CK_{\L, r}^+ = \CE_{\L, r} \cdot T_{0+:r}$ and hence $\CK_{\L, r}^+ / \CE_{\L, r} = T_{0+:r}/ \CT_{\L, r}$.

Finally we record a lemma, that guarantees that the subgroups attached to $\phi$ above --and hence also all constructions in the rest of the article-- are independent of the choice of the (tamely ramified) splitting field $K/k$ of $T$. This seems to be well-known to experts, but we could not find a reference in the literature.

\begin{lemma}\label{lm:indep_of_splitting_field}
Let $K/k$ and $K'/k$ be two finite tamely ramified extensions splitting $T$. We can canonically identify $\Phi = \Phi(G_K,T_K) = \Phi(G_{K'},T_{K'})$. Then for each $r \in \BR_{>0}$, we have
\[\{\alpha \in \Phi \colon \phi\circ N_{K/k} \circ \alpha^\vee(K_r^\times) = 1 \} = \{\alpha \in \Phi \colon \phi\circ N_{K/k} \circ \alpha^\vee(K_r^{\prime\times}) = 1\}. \]
With other words, the set $R_r$ of \cite[Eq. (3.6.1)]{Kaletha_19}, and hence also all Levi subgroups $G^i$, are independent of the choice of the splitting field $K/k$ of $T$.
\end{lemma}
\begin{proof}
The compositum of $K$ and $K'$ is still finite tame over $k$ and splits $T$. Thus, replacing $K'$ by this compositum, we may assume that $K \subseteq K'$. Then, as $\alpha^\vee \circ \Nm_{\BG_m,K'/K} = \Nm_{T,K'/K} \circ \alpha^\vee$, the statement follows from the transitivity of norm maps and the fact that $N_{K'/K}(K_r^{\prime\times}) = K_r^\times$. The latter equality holds because $r>0$ and $K'/K$ is tamely ramified. Indeed, it suffices to show this when $K'/K$ is unramified resp. totally ramified. In the first case, the claim follows from \cite[Cor. \!to Prop. 3 of Chap. V]{Serre_LF} and in the second case it follows from \cite[Cor. 3 to Prop. 5 of Chap. V]{Serre_LF}.
\end{proof}

\subsection{Reducing modulo $\CE_{\L, r}$}\label{sec:red_mod_CE}
We introduce the following convenient notation. By fixing $r$ and a Howe factorization $(\phi_{-1}, \L)$ of $\phi$ as in \S\ref{sec:Howe}, we put
\[\widetilde\BK = \widetilde\CK_{\L, r} / \CE_{\L, r} \quad \text{ and } \quad {}^E \widetilde\BK = {}^E \widetilde\CK_{\L, r} / {}^E\CE_{\L, r}.\]  We write $\BK, \BH, \BL, \BT, \BT^s$ (with $0 \le s \le r \in \widetilde\BR$) for the natural images of $\CK_{\L,r}, \CH_{\L, r}, L_r,  T_r, T_{s:r}$ in $\widetilde\BK$. The subgroups ${}^E\BH$, ${}^E\BL$, ${}^E\BT$ and ${}^E\BT^s$ are defined in a similar way. Note that by \Cref{lm:galois_coho_vanishes} we have $({}^E ?)^\tau = ?$ for $? \in \{\BH, \BL, \BT,$ $\BT^s, \dots \}$.

We note that $\BH/\BK^+$ and ${}^E\BH/{}^E\BK^+$  are linear $\ov\BF_q$-spaces and we denote them by
\[
V = \BH/\BK^+ \quad \text{ and } \quad {}^E V = {}^E\BH/{}^E\BK^+.
\]

It will also be convenient to write
\[
\hat V = \BH \quad \text{ and } \quad {}^E \hat V = {}^E\BH.
\]
indicating that we treat elements of $\BH$ as lifts of $V$.
% $\BI, \BU^0,$ $\bar\CI_{\phi, r}, \bar U^0_r,$

\subsection{Summary of setup}\label{sec:summary_setup} Attached to the character $\phi$ we have the following groups. The subgroup $\BH \subseteq \BK$ is the unipotent radical of $\BK$, we have $\BK/\BH \cong L_0$ and $\BK = \BL \BH$. We also have a short exact sequence
\begin{equation}\label{eq:extension_Heisenberg}
0 \to \BK^+ \to \hat V = \BH \to V = \BH/\BK^+ \to 0,
\end{equation}
of $\ov\BF_q$-groups, with $\BK^+ \cong \BT^{0+} \cong T_{0+:r}/\CT_{\phi,r}$.
This is an extension of abelian groups, so that the commutator pairing of $\BH$ induces a $\BK^+$-valued symplectic pairing on $V$. We study this pairing in \S\ref{sec:Heisenberg} below. We also have a similar situation, including sequence \eqref{eq:extension_Heisenberg}, for ${}^E\BK$ and its respective subgroups.

\section{Deligne--Lusztig construction in a Heisenberg group}\label{sec:Heisenberg}
We keep the setup from \S\ref{sec:tame_tori} - \ref{sec:summary_setup}.
The groups $\hat V = \BH$ and ${}^E\hat V = {}^E\BH$ are non-canonically isomorphic to Heisenberg groups. We exhibit an explicit isomorphism by using Yu's splitting of the sequence \eqref{eq:extension_Heisenberg} for ${}^E\hat V$, and proving that it is $\tau$-equvariant. In \S\ref{subsec:variety} we define a Deligne--Lusztig type subvariety for the group $\BK$.

\subsection{An isomorphism of Heisenberg groups}\label{sec:isom_Heisenberg}
%Set
%\[
%{}^E V = {}^E \CH_{\phi, r} / {}^E \CK^+_{\phi, r} \quad \text{ and } \quad {}^E \hat V = {}^E \CH_{\phi, r} / {}^E \CE_{\phi, r}.
%\]
Let
\[ \pi: {}^E \hat V \to {}^E V\] denote the projection homomorphism. Note that ${}^E V$ is a linear space over $\ov\BF_q$. The center of ${}^E\hat V$ is ${}^E \BK^+ \cong {}^E T_{0+:r} / {}^E \CT_{\phi, r}$.
% Let ${}^E Z = {}^E \CK_{\phi, r}^+ / {}^E \CE_{\phi, r}  \cong {}^E T_{0+:r} / {}^E T_{\phi, r}$ be the center of ${}^E \hat V$.

Let $\a \in \Phi \sm \Phi_L$. We set $1 \le i_\a \le d$ such that $\a \in \Phi_{G^i} \sm \Phi_{G^{i-1}}$ and put $r_\a = r_{i_\a - 1}$, where $\Phi_{G^i} = \Phi(G^i_E, T_E)$ is the root system of $G^i$ for $0 \le i \le d$. We set $({}^E\BK^+)_\a \subseteq {}^E\BK^+$ to be the natural image of $\a^\vee(E_{r_\a}^\times)$. Denote by ${}^E V_\a \subseteq {}^E V$ and ${}^E \hat V_\a \subseteq {}^E \hat V$ the subgroups generated by the natural images of the affine root subgroups ${}^E G^f$ such that $\a_f = \a$ and $f(\bx) \ge r_\a/2$ respectively (note that both are $\ov\BF_q$-vector spaces). Note that $\pi$ restricts to an isomorphism of linear spaces \[{}^E \hat V_\a \cong {}^E V_\a,\] which are both isomorphic to the linear space $(G^\alpha)_{\frac{r_\a}{2}: \frac{r_\a}{2}}$ of dimension $\le 1$.

Notice that $[{}^E \hat V_\a, {}^E \hat V_\b] = \d_{\a, -\b} ({}^E\BK^+)_\a$ for any $\a, \b \in \Phi$ (Kronecker delta). Hence the map $(x, y) \mapsto [x, y] = xyx^{-1}y^{-1}$ induces a skew symmetric pairing \[\k: {}^E \hat V \times {}^E \hat V \to {}^E\BK^+.\] It descends naturally to a pairing on ${}^E V$, which we still denote by $\k$. Then we have \[{}^E V \cong \bigoplus_{\a \in \Phi} {}^E V_\a, \quad\ {}^E V \cong {}^E\BK^+ \prod_{\a \in \Phi} {}^E \hat V_\a,\] where the product is taken with respect to any fixed order.

Let \[{}^E V^\sharp = {}^E V \times {}^E\BK^+\] be the Heisenberg group associated to $({}^E V, \k)$, whose multiplication law is given by \[(x, a)(y, b) = (x+y, a+b+\frac{1}{2}\k(x, y)), \text{ for }(x, a), (y, b) \in {}^E V \times {}^E\BK^+.\] Note that $L_r$ acts on ${}^E V^\sharp$ by $g: (x, a) \mapsto (\ad(g)x, a)$ for $g \in L_r$ and $(x, a) \in {}^E V^\sharp$.

Let $D \subseteq \Phi \sm \Phi_{L}$. We set \[ {}^E V_D = \bigoplus_{\a \in D} {}^E V_\a \quad \text{ and } \quad {}^E \hat V_D = \prod_{\a \in D} {}^E \hat V_\a,\]
where in the latter case we assume that $D \cap -D = \varnothing$, so that ${}^E\hat V_D$ is commutative subgroup of ${}^E \hat V$. Moreover, if $D$ is $\tau$-invariant, ${}^E V_D$ is too, and we set $V_D = ({}^E V_D)^\tau$. We warn the reader that we do not yet define $\hat V_D$, this will only happen in \S\ref{subsec:variety}; moreover, $\hat V_D$ will not be equal to $({}^E \hat V_D)^\tau$.

Let $P = L U_P \subseteq G$ be a parabolic subgroup (defined over $K$) with Levi subgroup $L \subseteq G$ and unipotent radical $U_P$. Denote by $\Phi_{U_P} \subseteq \Phi$ the set of roots appearing in $U_P$. Let $W_1 = {}^E \hat V_{\Phi_{U_P}}$ and $W_2 = {}^E \hat V_{-\Phi_{U_P}}$. Then $W_1 \cap {}^E\BK^+ = W_2 \cap {}^E\BK^+= \{0\}$ and each element of ${}^E \hat V$ has a unique expression $w_1 w_2 z$ with $w_1 \in W_1$, $w_2 \in W_2$ and $z \in {}^E\BK^+$.
\begin{proposition} \label{lift}
    There is a both ${}^E L_{[\bx]}$-equivariant (with respect to the adjoint actions) and  $\Gal(E/\brk)$-equivariant group isomorphism \[j: {}^E \hat V \stackrel{\sim}{\to} {}^E V^\sharp\] given by $w_1 w_2 z \mapsto (\pi(w_1 w_2), z + \frac{1}{2}\k(w_1, w_2))$ for $w_1 \in W$, $w_2 \in W_2$ and $z \in {}^E\BK^+$.
\end{proposition}
\begin{proof}
    It is proved in \cite[Lemma 10.1 \& 10.2]{Yu_01} that $j$ is an ${}^E L_{[\bx]}$-equivariant group homomorphism. We show $j$ is $\Gal(E/k)$-equivariant.

    Let $\t \in \Gal(E/k)$, $w_1 \in W_1$, $w_2 \in W_2$ and $z \in {}^E\BK^+$.  Noticing that $\t ({}^E \hat V_\a) = {}^E \hat V_{\t(\a)}$ for $\a \in \Phi \sm \Phi_{L}$, we have $\t(w_1) = x_1 x_2$ and $\t(w_2) = y_1 y_2$ for some $x_1, y_1 \in W_1$ and $x_2, y_2 \in W_2$ such that $\k(x_1, x_2) = \k(y_1, y_2) = 0$. Therefore, \begin{align*}
        j(\t(w_1 w_2 z)) &= j(x_1 x_2 y_1 y_2 \t(z)) \\ &= j((x_1 y_1) (x_2 y_2) (\t(z) + \k(x_2, y_1))) \\ &=(\pi(x_1 y_1 x_2 y_2), \t(z) + \frac{1}{2}(\k(x_1, y_2) + \k(x_2, y_1))) \\ &=(\pi(\t(w_1 w_2)), \t(z) + \frac{1}{2} \k(x_1x_2, y_1 y_2))) \\  &=(\pi(\t(w_1 w_2)), \t(z) + \frac{1}{2} \k(\t(w_1), \t(w_2)) \\ &= (\t(\pi(w_1 w_2)), \t(z) + \frac{1}{2} \t(\k(w_1, w_2))) \\ &= \t(j(w_1 w_2 z)),
    \end{align*} where the third equality follows from that $\k(x_1, x_2) = \k(y_1, y_2) = 0$; the fourth follows from that $\k(x_1, y_1) = \k(x_2, y_2) = 0$ and the sixth follows from that $\t$ commutes with $\pi$ and $\k$. The proof is finished.    
\end{proof}

\subsection{A Lagrangian subspace} \label{subsec:lagrangian}
In the rest of the paper, we assume further that $T$ is a maximally unramified maximal torus of $L$ in the sense of \cite[\S 3.4.1]{Kaletha_19}. (This is the case if $(T,\phi)$ is a tame elliptic regular pair in the sense of \cite[Definition 3.7.5]{Kaletha_19}.) In particular, there exists a $\t$-fixed Borel subgroup $B = T U$ of $L$ with unipotent radical $U$.

Recall that $\t$ is a generator of $\Gal(E / \brk)$ and that $\BK^+ = ({}^E\BK^+)^\t$, $V = ({}^E V)^\t$ and $\hat V = ({}^E \hat V)^\tau$. Note that $\kappa \colon {}^E V \times {}^E V \to {}^E\BK^+$ restricts to a pairing $V \times V \to \BK^+$, which we again denote by $\kappa$.

Let $\CC$ be the set of $\t$-orbits $C$ of $\Phi \sm \Phi_{L}$ such that $V_C \neq 0$, where $V_C = ({}^E \hat V_C)^\tau$. We have \[V \cong \bigoplus_{C \in \CC} V_C, \quad\ \hat V \cong \BK^+\prod_{C\in\CC} \hat V_C.\]
\begin{lemma} \label{V_C}
    Let $C$ be a $\t$-orbit of $\Phi \sm \Phi_L$. We have

    (1) $\dim V_C \le 1$;

    (2) $V_C \neq 0$ if and only if ${}^E V_C \neq 0$ and $\t^{|C|}$ acts on ${}^E V_C$ trivially;

    (3) $\k(V_C, V_{C'}) = 0$ if $C \neq -C'$;
    
    (4) $V_C = 0$ if $C = -C$;

    (5) $V_C \neq 0$ if and only if $V_{-C} \neq 0$, and in this case, $\k(V_C, V_{-C}) \neq 0$.
\end{lemma}
We note that on the $F$-invariant subspaces, \Cref{V_C}(4) is proven in \cite[p.66]{dBS_stability}.
\begin{proof}
Let $\alpha \in C$, so that $C = \{\tau^i(\alpha)\}_{i=0}^{|C|-1}$. By \eqref{eq:action_tau_affine_root_subgroups} the action of $\tau$ on ${}^E V_C$ is given by $(x_i)_{i=0}^{|C|-1} \mapsto (c_{i-1}x_{i-1})_{i=0}^{|C|-1}$ where $c_i \in \ov\BF_q^\times$ is some constant. Parts (1) and (2) follow directly from this description. For (3) it suffices to show that $\kappa({}^E V_C,{}^EV_{C'})=0$. We may assume that ${}^EV_C \neq 0 \neq {}^EV_{C'}$, otherwise there is nothing to prove. It is enough to check that $\kappa({}^E V_\alpha,{}^EV_{\alpha'})=0$ for any $\alpha \in C$, $\alpha' \in C'$. Let $f$ (resp. $f'$) be the unique affine root over $\alpha$ (resp. $\alpha'$) with $f(\bx) = \frac{r_\alpha}{2}$ (resp. $f'(\bx) = \frac{r_\alpha}{2}$). Then the image of $[{}^EG^f,{}^EG^{f'}]$ in $\CK^+_{\L,r}$ is contained in the natural image of ${}^EG^{f+f'}$. Hence its image in ${}^E\BK^+ \cong {}^ET_{0+:r}/{}^E\CT_{\L,r}$ vanishes. This proves (3). To prove (4) assume $V_C \neq 0$ and $C=-C$. Let $\alpha \in C$. Write $s = r_\alpha/2$. By assumption there is a (unique) affine root $f = (\alpha,n_f)$ over $\alpha$ such that $f(\bx) = s$. Then there is some $N>0$ such that $\tau^N(\alpha) = -\alpha$. As $f(\bx) = s = r_\alpha/2$, \Cref{lm:relation_between_constants} implies $\zeta^{eNr_\alpha} = 1$.

%and $\tau^N(\alpha) = -\alpha$, we have $n_f + n_{\tau^N(f)} = 2s = r_\alpha$. Then \Cref{lm:relation_between_constants}(1) gives for $x \in \ov\BF_q$:
%\begin{align*}
%\tau^{2N}(u_f(x)) &= \tau^N(u_{\tau(f)}(\bar a_\alpha^N \zeta^{Nen_f}x)) \\ 
%&= u_f(\bar a_{-\alpha}^N \bar a_\alpha^N \zeta^{Ne(n_f + n_{\tau(f)})}x) \\
%&= u_f(\zeta^{Ne(n_f + n_{\tau(f)})}x) \\
%&= u_f(\zeta^{Ner_\alpha}x)
%\end{align*}
%where the last equation follows from \Cref{lm:relation_between_constants}. On the other hand, as $V_C \neq 0$, the induced action of $\tau^{2N}$ on the affine root subgroup $U_f(\ov\BF_q)$ is trivial. Combining these observations, we deduce $\zeta^{Ner_\alpha} = 1$. 

Now let $x \in \ov\BF_q$, so that $1+\pi^{er_\alpha} \in E_{r_\alpha:r_\alpha}^\times$. Then we compute in $T_{r_\alpha:r_\alpha}$:
\begin{align*}
&\quad\ \Nm_{K/k_f} (\alpha^\vee  \otimes 1+ \pi^{er_\alpha} x) \\
&= \sum_{i=0}^{2N-1} \tau^i(\alpha^\vee \otimes 1+ \pi^{er_\alpha} x) \\
&= \sum_{i=0}^{N-1} \tau^i(\alpha^\vee \otimes 1+\pi^{er_\alpha} x) + \sum_{i=0}^{N-1} \tau^{i+N}(\alpha^\vee \otimes 1+\pi^{er_\alpha} x) \\
&= \sum_{i=0}^{N-1} \tau^i(\alpha^\vee \otimes 1+\pi^{er_\alpha} x) + \sum_{i=0}^{N-1} \tau^i(-\alpha^\vee \otimes 1+\zeta^{Ner_\alpha} x) \\
&= \sum_{i=0}^{N-1} \tau^i(\alpha^\vee \otimes 1+\pi^{er_\alpha} x) - \sum_{i=0}^{N-1} \tau^i(-\alpha^\vee \otimes  1+ \pi^{er_\alpha}x) \\
&= 0,
\end{align*}
as $\zeta^{Ner_\alpha} = 1$. With other words, $\Nm_{K/k_f}\alpha^\vee(E_{r_\alpha}^\times) = 1$. From the factorization $\Nm_{K/k} = \Nm_{k_f / k} \circ \Nm_{K/k_f}$,  we deduce that also
\[ 
\Nm_{K/k} \circ \alpha^\vee (E_{r_\alpha}^\times) = 1,
\]
and hence also $\phi \circ \Nm_{K/k} \circ \alpha^\vee (E_{r_\alpha}^\times) = 1$, which contradicts the definition of $r_\alpha = r_{i_\alpha - 1}$. This proves (4). The first part of (5) follows from (2), the formula \eqref{eq:action_tau_affine_root_subgroups} and \Cref{lm:relation_between_constants}. For the second part of (5), suppose $V_C \neq 0$ and let $C = \{\tau^i(\alpha)\}_{i=0}^{|C|-1}$. Write $s = r_\alpha/2$. Let $f = (\alpha,n_f)$, resp. $f' = (-\alpha,n_{f'})$ be the unique affine root with $f(\bx) = s$, resp. $f'(\bx) = s$. Then ${}^EV_C \simeq \bigoplus_{i=0}^{|C|-1}{}^EG^{\tau^i(f)}$ is a product of root subgroups. As $V_C \neq 0$ by assumption, \Cref{lm:relation_between_constants}(1) implies that under this isomorphism $V_C$ precisely consists of tuples \[ X = (x, \, \zeta^{en_f}\bar a_\alpha x , \, \zeta^{e(n_f+n_{\tau(f)})} \bar a_\alpha \bar a_{\tau(\alpha)} x \, , \, \dots, \, \zeta^{e\sum_{j=0}^{|C|-2}n_{\tau^j(f)}} \cdot \prod_{j=0}^{|C|-2} \bar a_{\tau^j(\alpha)} \cdot x )\]
with arbitrary $x \in \ov\BF_q$. Moreover, by part (2) we see that 
\begin{equation}\label{eq:tauC_trivial} 
\zeta^{e\sum_{j=0}^{|C|-1} n_{\tau^j(f)}}\prod_{j=0}^{|C|-1}\bar{a}_{\tau^j(\alpha)} = 1.
\end{equation}
Let $Y \in V_{-C}$ is a similar tuple with $x$ replaced by $y$ and $f$ by $f'$ everywhere. Taking the product of \eqref{eq:tauC_trivial} with the similar expression for $\alpha$ and using $V_C \neq 0$ along with \Cref{lm:relation_between_constants}, we deduce $\zeta^{er_\alpha |C|} = 1$. Now, we compute in $T_{r_\alpha:r_\alpha}$:
\begin{align*}
[X,Y] &= \sum_{i=0}^{|C|-1} \tau^i(\alpha)^\vee \otimes 1 + \pi^{er_\alpha} x y \zeta^{e\sum_{j=0}^{i-1} (n_{\tau^j(f)} + n_{\tau^j(f')})} \cdot \prod_{j=0}^{i-1} \bar a_{\tau^j(\alpha)} \bar a_{-\tau^j(\alpha)} \\
&= \sum_{i=0}^{|C|-1} \tau^i(\alpha)^\vee \otimes 1 + \pi^{er_\alpha} x y \zeta^{ier_\alpha }.
\end{align*}
%\begin{align*}
%[X,Y] &= \sum_{i=0}^{|C|-1} \alpha^\vee \otimes 1 + \pi^{er_\alpha} x_f y_{f'} \sum_{i=-1}^{|C|-2} \zeta^{e\sum_{j=0}^i (n_{\tau^j(f)} + n_{\tau^j(f')})} \cdot \prod_{j=0}^i \bar a_{\tau^j(\alpha)} \bar a_{-\tau^j(\alpha)} \\
%&= \sum_{i=0}^{|C|-1} \alpha^\vee \otimes 1 + \pi^{er_\alpha} x_f y_{f'} \sum_{i=-1}^{|C|-2} \zeta^{e\sum_{j=0}^i (n_{\tau^j(f)} + n_{\tau^j(f')})}.
%\end{align*}
where the second equality follows from \Cref{lm:relation_between_constants}(2) and that $n_{\tau^j(f)} + n_{\tau^j(f')} = r_\alpha$. It thus suffices to show that there exist some $x,y \in \ov\BF_q$ for which this expression does not vanish in $T_{r_\alpha:r_\alpha}$. By definition of $r_\alpha$, we have $\phi \circ \Nm_{K/k} \circ \alpha^\vee (K_{r_\alpha}^\times) \neq 1$ and $\phi \circ \Nm_{K/k} \circ \alpha^\vee (K_{r_\alpha+}^\times) = 1$. Noting that $\phi \circ \Nm_{K/k} \circ \alpha^\vee$ factors through $\Nm_{K/k_f} \circ \alpha^\vee$, this implies that there exists some $z \in \ov\BF_q$ with $\sum_{\lambda=0}^{e-1} \tau^\lambda(\alpha^\vee \otimes 1 + \pi^{er_\alpha}z) \neq 1$ in $T_{r_\alpha:r_\alpha}$. Put $d:=\frac{e}{|C|}$. We compute in $T_{r_\alpha:r_\alpha}$
\begin{align*}
1 \neq \sum_{\lambda=0}^{e-1} \tau^\lambda(\alpha^\vee \otimes 1 + \pi^{er_\alpha}z) &= \sum_{i=0}^{|C|-1} \sum_{i'=0}^{d-1} \tau^{i+i'|C|}(\alpha^\vee \otimes 1 + \pi^{er_\alpha}z) \\
&= \sum_{i=0}^{|C|-1} \tau^i(\alpha^\vee \otimes 1+ \pi^{er_\alpha} \sum_{i'=0}^{d-1} \zeta^{er_\alpha |C|i'} z) \\
&= \sum_{i=0}^{|C|-1} \tau^i(\alpha^\vee \otimes 1+ \pi^{er_\alpha} d z) \\
&= \sum_{i=0}^{|C|-1} \tau^i(\alpha)^\vee \otimes 1 + \pi^{er_\alpha}  \zeta^{ier_\alpha} dz
\end{align*}
where the second equality follows from $\tau^{|C|}(\alpha) = \alpha$, the third equality from $\zeta^{er_\alpha|C|} = 1$, and the last equality is the action of $\tau$ on the affine coroots. Note that $d$ is invertible in $\ov\BF_q$ (as $e$ is), and pick any $x,y$ with $xy=dz$. This finishes the proof of (5).

By the same computation as in \eqref{eq:2s_equal_fplustNf} we have  $n_{\tau^j(f)}+ n_{\tau^j(f')} = 2s = r_\alpha$ for any $j$. Thus
\begin{align*}
[X,Y] &= \sum_{i=0}^{|C|-1} \alpha^\vee \otimes 1 + \pi^{er_\alpha} x_f y_{f'} \sum_{i=0}^{|C|-1} \zeta^{ier_\alpha}
\end{align*}
where the second equation is by \Cref{lm:relation_between_constants}(2).
\end{proof}

Recall that $\CC$ is the set of $\tau$-orbits $C$ in $\Phi \sm \Phi_L$ such that $V_C \neq 0$. Set $\Psi = \Psi_\phi = \cup_{C \in \CC} C$. Recall that $B = TU \subseteq L$ is a Borel subgroup with unipotent radical $U$.

\begin{proposition} \label{lagrangian}
    There exists a subset $\Psi_+  \subseteq \Psi$ such that $\Psi = \Psi_+  \cup -\Psi_+$, $\Psi_+ \cap -\Psi_+ = \varnothing$ and $V_{\Psi_+}$ is normalized by $B_0$.
\end{proposition}
\begin{proof}
   By \Cref{V_C} (4), it suffices to construct inductively a sequence of subsets $\varnothing = \Psi_{0,+} \subseteq \Psi_{1,+} \subseteq \cdots \subseteq \Psi_{|\CC|/2,+} \subseteq \Psi$ such that $\Psi_{i,+} \sm \Psi_{i-1,+}$ is a $\t$-orbit of $\Psi$, $\Psi_{i,+} \cap -\Psi_{i,+} = \varnothing$ and $V_{\Psi_{i,+}}$ is normalized by $B_0 \subseteq L_0$ for $1 \le i \le |\CC_i|/2$.

   Indeed, we put $\Psi_{0,+} = \varnothing$ and suppose $\Psi_{i,+}$ is already constructed with the desired properties. Suppose $i < |\CC|/2$ and we construct $\Psi_{i+1, +}$ as follows. Let \[V_{\Psi_{i,+}}^\perp := \{v \in V; \k(v, V_{\Psi_{i,+}}) = 0\} = V_{\Psi \sm -\Psi_{i,+}},\] where the second equality follows from \Cref{V_C}. By induction, $V_{\Psi_{i,+}}$ is $B_0$-invariant. Hence $V_{\Psi_{i,+}}^\perp$ is also $B_0$-invariant. 

   Let $\Phi_L^+ \subseteq \Phi_L$ be the set of positive roots determined by $B$.  For $\a, \b \in \Phi$ we write $\a \le \b$ if $\b - \a$ is a sum of roots in $\Phi_L^+$. Let $\a_{i+1}$ be a maximal root in $\Psi \sm (\Psi_{i,+} \cup -\Psi_{i,+})$ with respect to the partial order $\le$. Let $C_{i+1}$ be the $\t$-orbit of $\a_{i+1}$, and set $\Psi_{i+1,+} = \Psi_{i,+} \cup C_{i+1}$. Consider the following set \[D_{i+1} := \{\a \in \Phi; \a \ge \g \text{ for some } \g \in C_{i+1}\} = \{\a \in \Phi; \t^i(\a) \ge \a_{i+1} \text{ for some } i \in \BZ\},\] where the second equality follows from that $\Phi_L^+$ is $\t$-invariant. By the choice of $\a_{i+1}$, we have \[(\Psi \sm -\Psi_{i,+}) \cap D_{i+1} \subseteq \Psi_{i+1,+}.\] By definition, \[\Ad(B_0)(V_{C_{i+1}}) \subseteq V_{\Psi_{i,+}}^\perp \, \cap \,{}^E V_{D_{i+1}} = V_{\Psi \sm -\Psi_{i,+}} \,\cap \,{}^E V_{D_{i+1}} =  V_{(\Psi \sm -\Psi_{i,+}) \cap D_{i+1}} \subseteq V_{\Psi_{i+1,+}}.\] Thus $V_{\Psi_{i+1, +}}$ is $B_0$-invariant, and the induction procedure is finished.
\end{proof}

\subsection{A Deligne--Lusztig type construction} \label{subsec:variety}
Let notation be as in \Cref{subsec:lagrangian}. Let $\Psi_+$ be as in \Cref{lagrangian}. 

Recall the isomorphism $j$ from \Cref{lift}. By taking $\tau$-invariants, it induces an isomorphism $j \colon \hat V \stackrel{\sim}{\to} V^\sharp$. Now, $V^\sharp \cong V \times \BK^+$ contains a \emph{canonical} copy $V \times \{0\}$ of $V$. For any $\tau$-equivariant subset $D \subseteq \Phi \sm \Phi_L$ we set $\hat V_D = j\i(V_D \times \{0\}) \subseteq \hat V$. By \Cref{lift} and \Cref{lagrangian}, $\hat V_{\Psi_+}$ is a commutative group normalized by $B_0 = T_0 U_0$ and hence the product
\[\BI := U_0 \hat V_{\Psi_+} \subseteq \BK\] is subgroup. Then the attached deep level Deligne--Lusztig variety is defined by \[Z = Z_{\L, B, \Psi_+, r} = \{g \in \BK; g\i F(g) \in F\BI\}.\] As $\BI$ is normalized by $\BT$, the variety $Z$ is endowed with a natural action of $\BT^F$ given by $t: x \mapsto x t$. Let \[H_c^i(Z, \ov\BQ_\ell)[\phi] \subseteq H_c^i(Z, \ov\BQ_\ell)\] be the subspace of the $i$th $\ell$-adic cohomology of $Z$ on which $\BT^F$ acts via the character $\phi$.

Note that $Z$ admits a second action of $\BK^F$ given by $g: x \mapsto g x$, which commutes with previous action of $\BT^F$. Therefore, the weight spaces $H_c^i(Z, \ov\BQ_\ell)[\phi]$ are natural representations of $\BK^F$. We define the following virtual $\BK^F$-module \[\CR_{\BT}^{\BK}(\phi):= H_c^*(Z, \ov\BQ_\ell)[\phi] := \sum_i (-1)^i H_c^i(Z, \ov\BQ_\ell)[\phi].\] By inflation, we also implicitly view $\CR_{\BT}^{\BK}(\phi)$ as a virtual $\CK_\L^F$-module.

\subsection{An extension} \label{subsec:extension}
As $T$ is elliptic in $G$, we have a natural adjoint action of $T^F$ on $\BK$ induced by $s: x \mapsto s x s\i$. Moreover, this action preserves subgroup $\BI$ and hence the variety $Z$. 

Following \cite[\S 3.4]{Kaletha_19}, we consider the action of $\CK_\L^F \rtimes T^F$ on $Z$ given by $g \rtimes z : x \mapsto g z x z\i$. This action commutes with the previous action of $T_{\bx, 0}^F$ on $Z$ by right multiplication. Thus the weight spaces $H_c^i(Z, \ov\BQ_\ell)[\phi]$ are natural representations of $\CK_\L^F \rtimes T^F$. Consider the tensor products \[H_c^i(Z, \ov\BQ_\ell)[\phi] \otimes \phi,\] where $\phi$ is viewed as a character of $\CK_\L^F \rtimes T^F$ via the natural projection $\CK_\L^F \rtimes T^F \to T^F$. As $T \subseteq L$ is a maximally unramified maximal torus, we have $\CK_\L^F \cap T^F = L_{\bx, 0}^F \cap T^F = T_{\bx, 0}^F$. Then by definition the action $z \rtimes z\i$ for $z \in T^F$ on $H_c^i(Z, \ov\BQ_\ell)[\phi] \otimes \phi$ is trivial. Therefore, $H_c^i(Z, \ov\BQ_\ell)[\phi] \otimes \phi$ descends to a representation of $\CK_\L^F T^F$. We define \[\hat\CR_{\BT}^{\BK}(\phi):= H_c^*(Z, \ov\BQ_\ell)[\phi] \otimes \phi\] as a virtual $\CK_\L^F T^F$-module. Note that $\hat\CR_{\BT}^{\BK}(\phi) |_{\CK_\L^F} = \CR_{\BT}^{\BK}(\phi)$.

\section{Mackey formula} \label{sec:Mackey}
Let the notation be as in \S\ref{subsec:variety}. Recall that $B = T U$ is a Borel subgroup of $L$, and we write $\BL$, $\BU$ and $\BT$ for the natural images of $L_r$, $U_r$ and $T_r$ in $\BK$ respectively.
% We put $\BK = \CK_{\phi, r} / \CE_{\phi, r}$. For a subset $R \subseteq \CK_{\phi, r}$, write $\bar R$ for the its natural image in $\BK$. Let $L = G^0$. We write $\BH, \BI, \BL, \BU^0, \BT, \BT^{r'}$ instead of $\bar\CH_{\phi, r}, \bar\CI_{\phi, r}, \bar L_r, \bar U^0_r, \bar T_r, \bar T_{r':r}$ and so on.
Let $N_{\BL}(\BT) = \{x \in \BL; {}^x \BT = \BT\}$ and $W_{\BL}(\BT) = N_{\BL}(\BT) / \BT$. Then  \[\BL = \sqcup_{w \in W_{\BL}(\BT)} \BU \dw \BT \BU,\] where $\dw$ is a lift of $w$ in $N_{\BL}(\BT)$.

Let $\Psi = \Psi_+ \sqcup -\Psi_+$ be as in \Cref{lagrangian}. Put $\Psi_- = - \Psi_+$ and we have \[\BH = \hat V = \hat V_{\Psi_+} \hat V_{\Psi_-} \BT^{0+}.\] Thus \[\BK = \BH \BL = \bigsqcup_{w \in W_{\BL}(\BT)} \BU \hat V \dw \BT \BU = \bigsqcup_{w \in W_{\BL}(\BT)} \BI \hat V_{\Psi_- \cap {}^w \Psi_-} \dw \BT \BI.\]

Let $C \subseteq \Psi$. In particular, $V_C \neq 0$. We set $i_C = i_\a$ (see \S\ref{sec:isom_Heisenberg}) and $r_C = r_\a$ for some/any $\a \in C$. Define \[(\BK^+)_C = \k(V_C, V_{-C}) \subseteq \BT^{r_C},\] which is a one dimensional $\ov\BF_q$-linear space by \Cref{V_C}.

\subsection{Auxiliary results} \label{subsec:D_i}
Let $D_1, D_2 \subseteq \Psi$ be two $\t$-stable subsets such that $D_i \cap -D_i = \varnothing$ for $i=1, 2$. Let $\dw \in \BL^F$ which normalizes $\BT$. The map $(y_1, y_2) \mapsto y_1 y_2$ gives an affine space fibration \[\g: \hat V_{D_1} \times \hat V_{D_2} \to \hat V_{D_1} \hat V_{D_2} \cong \hat V_{D_1 \sm D_2} \hat V_{D_2},\] whose fibers are isomorphic to $\hat V_{D_1 \cap D_2}$. Consider the following varieties \begin{align*}
    X = X_{D_1, D_2, w} &= \{(x_1, x_2, y_1, y_2, z) \in \hat V_{D_1} \times \hat V_{D_2} \times \hat V_{D_1} \times \hat V_{D_2} \times \BT;  {}^\dw(x_1 x_2) L(z) = y_1 y_2\}; \\ Y = Y_{D_1, D_2, w} &= \{(x, z) \in (\hat V_{D_1} \hat V_{D_2}) \times \BT; {}^\dw(x) L(z) \in \hat V_{D_1} \hat V_{D_2}\},
\end{align*} 
where $L: \BT \to \BT$ denotes the Lang's map given by $z \mapsto z\i F(z)$. Note that $\BT^F \times \BT^F$ acts on $X$ and $Y$ respectively by \begin{align*} (s, t) &: (x_1, x_2, y_1, y_2, z) \mapsto ({}^s x_1, {}^s x_2, {}^{w(s)} y_1, {}^{w(s)} y_2, w(s) z t) \\ (s, t) &: (x, z) \mapsto ({}^s x, w(s) z t).   \end{align*} As $\g$ is an affine space fibration, so is the $\BT^F \times \BT^F$-equivariant map \[X \to Y, \quad (x_1, x_2, y_1, y_2, z) \mapsto (x_1x_2, z).\] 
In particular, for another character $\psi$ of $\BT^F$ we have an isomorphism \[H_c^*(X, \ov\BQ_\ell)[\psi \boxtimes \phi] \cong H_c^*(Y, \ov\BQ_\ell)[\psi \boxtimes \phi].\]

Set $D_i' = D_i \cap {}^{w\i}(D_1 \cup D_2)$ for $i = 1, 2$. As ${}^w (D_1' \cup D_2') \subseteq D_1 \cup D_2$, we have ${}^\dw (\hat V_{D_1'} \hat V_{D_2'}) \subseteq \hat V_{D_1} \hat V_{D_2} \BT$. Let $\pr_T: \hat V_{D_1} \hat V_{D_2} \BT \to \BT$ be the natural projection. We consider the map $\d: \hat V_{D_1'} \hat V_{D_2'} \to \BT$ given by $x \mapsto \pr_T({}^\dw x)$. 

\begin{lemma} \label{square}
    We have the following Cartesian diagram \[ \xymatrix{
    Y \ar[d]_{\pr_1} \ar[r]^{\pr_2} &  \BT \ar[d]_{-L} \\
    \hat V_{D_1'} \hat V_{D_2'} \ar[r]^\d &  \BT,}\] where $\pr_1: Y \to \hat V_{D_1} \hat V_{D_2}$ and $\pr_2: Y \to \BT$ are the projections given by $(x, \t) \mapsto x$ and $(x, \t) \mapsto \t$ respectively.
\end{lemma}
\begin{proof}
    Let $(x, \t) \in Y \subseteq (\hat V_{D_1} \hat V_{D_2}) \times \BT$. By taking the natural projection $\pi: \hat V \to V$, we deduce by definition that \[\pi(x) \in V_{{}^{w\i}(D_1 \cup D_2)} \cap V_{D_1 \cup D_2} = V_{D_1' \cup D_2'}.\] As $\pi$ restricts to a bijection $\hat V_{D_1} \hat V_{D_2} \cong V_{D_1 \cup D_2}$, we have $x \in \hat V_{D_1'} \hat V_{D_2'}$. Then the statement follows by the definition of $Y$.
\end{proof}

We fix a linear order $\preceq$ on $\t$-orbits of $D_1' \cup D_2'$ such that $C_1 \prec C_2$ if $C_1 \subseteq D_1' \sm D_2'$ and $C_2 \subseteq D_2'$. Then we have an isomorphism of varieties \[ \chi: \prod_C \hat V_C \to \hat V_{D_1'} \hat V_{D_2'},\quad (x_C)_C \mapsto \prod_C x_C,\] where $C$ ranges over $\t$-orbits of $D_1' \cup D_2'$, and the product is taken with respect to $\preceq$.

We recall the following result from \cite[Proposition 2.10]{Boyarchenko_12}.
\begin{proposition} \label{iteration}
    Let $A$ be a connected commutative algebraic group and $\CL$ a multiplicative rank-one local system on $A$. Let $Z_1$ be a variety and let $\xi: Z = Z_1 \times \BG_a \to A$ be a morphism of the form \[(z, y) \mapsto \eta(z, y) \z(z)\] such that for any $z \in Z_1$ the morphism $\eta_z: \BG_a \to \bar A$ given by $y \mapsto \eta(z, y)$ is a group homomorphism. Then we have \[H_c^i(Z \sm Z', \xi^* \CL) = 0,\] where $Z' \subseteq Z$ is the closed subvariety consisting of points $(z, y) \in Z$ such that $\eta_z^* \CL$ is trivial.
\end{proposition}
\begin{proof}
Let $p \colon Z \to Z_1$ denote the natural projection. Then $\xi^\ast\CL \simeq \eta^\ast\CL \otimes p^\ast\zeta^\ast\CL$, as $\CL$ is multiplicative. By projection formula this implies $Rp_!(\xi^\ast\CL) \simeq Rp_!\eta^\ast\CL \otimes \zeta^\ast \CL$. It now suffices to show that $Rp_!\eta^\ast\CL$ restricted to $Z_1 \sm p(Z')$ is zero. Let $p_z \colon \{z\} \times \BG_a \to \{z\}$ denote the fiber of $p$ over $z \in Z$. By proper base change, $Rp_!(\eta^\ast\CL)_z = Rp_{z!}(\eta_z^\ast\CL)$, which is zero if $z \in Z \sm Z'$ by \cite[Lemma 9.4]{Boyarchenko_10}.
\end{proof}

Let $D = (D_2 \cap -D_1) \cap {}^{w\i}(D_1 \cap -D_2)$.
\begin{lemma} \label{delta}
    Using the isomorphism $\chi$ above, the map $\d: \hat V_{D_1'} \hat V_{D_2'} \to \BT$ is given by \[(x_C)_C \mapsto \sum_{C_2} {}^\dw [x_{-C_2}, x_{C_2}],\] where $C_2$ ranges over $\t$-orbits of $D$.
\end{lemma}
\begin{proof}
    By definition and that $[\hat V_{D_i}, \hat V_{D_i}] = 0$ for $i = 1, 2$, we have \[\d((x_C)_C) = \sum_{C_1, C_2} {}^\dw [x_{C_1}, x_{C_2}],\] where $C_1$ and $C_2$ range over $\t$-orbits of $D_1 \sm D_2$ and $D_2$ respectively such that $w(C_1) \subseteq D_2$ and $w(C_2) \subseteq D_1 \sm D_2$. The statement then follows by noticing that $[x_{C_1}, x_{C_2}] = 0$ unless $C_1 = -C_2$. 
\end{proof}

\begin{proposition} \label{dim}
    We have \[\dim_{\ov\BQ_\ell} H_c^*(X, \ov\BQ_\ell)[\psi\i \boxtimes \phi] = \dim_{\ov\BQ_\ell} H_c^*(Y, \ov\BQ_\ell)[\psi\i \boxtimes \phi] = \d_{{}^{w\i}\psi, \phi}.\]
\end{proposition}
\begin{proof}
First we consider the $\phi$-weight space $H_c^i(Y, \ov\BQ_\ell)[\phi]$ of $\{1\} \times \BT^F$ for $i \in \BZ$. By \Cref{square} we have \[H_c^i(Y, \ov\BQ_\ell)[\phi] = H_c^i(\hat V_{D_1'} \hat V_{D_2'}, \d^* \CL_\phi).\] Let $C$ be a $\t$-orbit of $\Psi$ and $u \in \hat V_{-C}$, the map \[\d_u : \hat V_C \cong \BG_a \to \BT, \quad\  v \mapsto [u, v]\] is a group homomorphism and the pull-back $\d_u^* \CL_\phi$ is nontrivial if and only if $u = 0$. By \Cref{iteration} and \Cref{delta} we have \[H_c^i(Y \sm Y', \ov\BQ_\ell)[\phi] \cong H_c^i((\hat V_{D_1'} \hat V_{D_2'}) \sm \hat V_{D_1' \sm -D} \hat V_{D_2'}, \d^* \CL_\phi) = 0,\] where $Y' = \{(x, \t)\in Y; x \in \hat V_{D_1' \sm -D} \hat V_{D_2'}\}  = \hat V_{D_1' \sm -D} \hat V_{D_2'} \times \BT^F$. In particular, \[\dim_{\ov\BQ_\ell} H_c^*(Y, \ov\BQ_\ell)[\psi\i \boxtimes \phi] = \dim_{\ov\BQ_\ell} H_c^*(Y', \ov\BQ_\ell)[\psi\i \boxtimes \phi].\] As the natural projection $Y' = \hat V_{(D_1' \cup D_2') \sm -D} \times \BT^F \to \BT^F$ is a $\BT^F \times \BT^F$-equivariant affine space fibration, we have \[\dim_{\ov\BQ_\ell} H_c^*(Y', \ov\BQ_\ell)[\psi\i \boxtimes \phi] = \dim_{\ov\BQ_\ell} H_c^*(\BT^F, \ov\BQ_\ell)[\psi\i \boxtimes \phi] = \d_{{}^{w\i} \psi, \phi}.\] The proof is finished.
\end{proof}

\subsection{Inner product computation}
Now we state Mackey formula for the inner product of the virtual $\BK^F$-module 
\begin{theorem} \label{product}
    We have \[\<\CR_{\BT}^{\BK}(\phi), \CR_{\BT}^{\BK}(\phi)\>_{\BK^F}  = \sharp \{w \in W_{\BL}(\BT)^F;  {}^w\phi = \phi\}.\] Here we identify $\phi$ with its restriction $\phi|_{\BT^F}$.
\end{theorem}
\begin{proof}
    For $w \in W_{\BL}(\BT)$ we set \[\Sigma_w = \{(x, x', v, \bar v, \t, u) \in F\BI \times F\BI \times \BI \times \hat V_{\Psi_- \cap {}^w \Psi_-} \times \BT \times \BI; x F(\bar v \dw \t) = v \bar v \dw \t u x'\}.\]
    Write $\Sigma_w = \Sigma_w' \sqcup \Sigma_w''$, where $\Sigma_w'$ and $\Sigma_w''$ are defined by the conditions $\bar v \neq 0$ and $\bar v = 0$ respectively. Note that $\BT^F \times \BT^F$ acts on $\Sigma_w$ by \[(s, t): (x, x', v, \bar v, \t, u) \mapsto (sxs\i, tx't\i, svs,  s \bar v s\i, \dw\i(s) \t t\i, tut\i).\] As in the proof of \cite[Proposition 5.5]{Nie_24}, it suffices to show that \[\dim H_c^*(\Sigma_w'', \ov\BQ_\ell)[\phi\i \boxtimes \phi] = \d_{w, F(w)} \cdot \d_{{}^w \phi, \phi}\] and $\dim H_c^*(\Sigma_w', \ov\BQ_\ell)[\phi\i \boxtimes \phi] = 0$.

    Let $H = \{(s, t) \in \BT \times \BT; s\i F\i(s) = \dw t\i F\i(t) \dw\i\}$, which acts on $\Sigma_w''$ by \[(s, t): (x, x', v, \t, u) \mapsto (sxs\i, tx't\i, svs, \dw\i(s) \t t\i, tut\i).\] Let $H_\red^\circ$ and $\BT_\red$ be the reductive parts of the identity component of $H^\circ$ and $\BT$ respectively. As $\BI^{\BT_\red} = \hat V_D$ with $D = \{\a \in \Psi; \a(\BT_\red) = \{1\}\}$. It follows that $(\Sigma_w'')^{H_\red^\circ} \neq \varnothing$ only if $F(w) = w$, and in this case we have \[(\Sigma_w'')^{H_\red^\circ} \cong X_{D, FD, w\i}\] as $\BT^F \times \BT^F$-varieties, where $X_{D, FD, w\i}$ is as in \S\ref{subsec:D_i}. It follows from \Cref{dim} that $\dim_{\ov\BQ_\ell} H_c^i((\Sigma_w'')^{H_\red^\circ}, \ov\BQ_\ell)[\phi\i \boxtimes \phi] = \d_{{}^w \phi, \phi}$. Hence $\dim H_c^*(\Sigma_w'', \ov\BQ_\ell)[\phi\i \boxtimes \phi] = \d_{w, F(w)} \cdot \d_{{}^w \phi, \phi}$ as desired.

    It remains to show $H_c^*(\Sigma_w', \ov \BQ_\ell)[\phi\i \boxtimes \phi] = 0$. Note that \[\hat V_{\Psi_- \cap {}^w \Psi_-} = \bigoplus_{C \in (-\CC) \cap (-{}^w\CC)} \hat V_C.\]  For $\bar v \in \hat V_{\Psi_- \cap {}^w \Psi_-}$ and $C \in (-\CC) \cap (-{}^w\CC)$ let $\bar v_C \in \hat V_C$ such that $\bar v = \sum_C \bar v_C$. We fix a total order $\le$ on $ (-\CC) \cap (-{}^w\CC)$. Let $\hat V_{\Psi_- \cap {}^w \Psi_-}^C$ be the subset of elements $\bar v$ such that $\bar v_C \neq 0$ and $\bar v_{C'} = 0$ for all $C' < C$. Then we have \[\hat V_{\Psi_- \cap {}^w \Psi_-} \sm \{0\} = \bigsqcup_C \hat V_{\Psi_- \cap {}^w \Psi_-}^C.\] The above decomposition induces a decomposition \[\Sigma_w' = \bigsqcup_C \Sigma_w^{ \prime, C}.\] It remains to show $H_c^*(\Sigma_w^{\prime, C}, \ov \BQ_\ell)[\phi\i \boxtimes \phi] = 0$ for all $C \in (-\CC) \cap (-{}^w\CC)$.

    Let $C \in (-\CC) \cap (-{}^w\CC)$. Consider the restricted action of $\BT^F \cong \BT^F \times \{1\} \subseteq \BT^F \times \BT^F$ on $\Sigma_w^{\prime, C}$ given by \[s: (x, x', v, \bar v, \t, u) \mapsto (sxs\i, x', svs\i, s\bar vs\i, w\i(s)\t, u).\]  It suffices to show the $\phi\i$-weight subspace $H_c^*(\Sigma_w^{\prime, C}, \ov \BQ_\ell)[\phi\i]$ is trivial.

    Let $\bar v \in \hat V_{\Psi_- \cap {}^w \Psi_-}^C$. We fix an isomorphism \[\l_{\bar v}: \hat V_{-C} \overset \sim \longrightarrow (\BK^+)_C, \quad \z \to \k(\bar v, \z).\] Consider the subgroup \[H = \{s \in \BT^{r_C}; s\i F\i(s) \in (\BK^+)_C\}.\] For $s \in H$ we define an isomorphism $f_s: \Sigma_w^{\prime, C} \to \Sigma_w^{\prime, C}$ by \[f_s: (x, x', v, \bar v, \t, u) \mapsto (x_s, x' F({}^{(\dw \t)\i} \z), svs\i, s\bar vs\i, w\i(s)\t, u)\] with $\z = \l_{\bar v}\i(s F\i(s)\i)$ such that \[x_s F(\bar v \dw \t) = s v \bar v \dw \t u x' F({}^{(\dw \t)\i} \z).\] The induced map of $f_s$ on each subspace $H_c^i(\Sigma_w^{\prime, C}, \ov \BQ_\ell)$ is trivial for each $s \in N_F^{F^n}((\BK^+)_C^{F^n}) \subseteq H^\circ$. Here $n \in \BZ_{\ge 1}$ such that $F^n(C) = C$, and $N_F^{F^n}: \BT^{r_C} \to \BT^{r_C}$ is the map given by $s \mapsto s F(s) \cdots F^{n-1}(s)$. On the other hand, we have \[\phi |_{N_F^{F^n}((\BK^+)_C^{F^n})} = \phi_{i_C - 1} |_{N_F^{F^n}((\BK^+)_C^{F^n})},\] which is nontrivial since $\phi_{i_C-1}$ is $(G^{i_C-1}, G^{i_C})$-generic. Thus $H_c^*(\Sigma_w^{\prime, C}, \ov \BQ_\ell)[\phi\i]$ is trivial as desired.
\end{proof}

\section{A concentration theorem} \label{sec:concentration}
In this section, we follow the approach in \cite[\S 5]{Nie_24} to prove a concentration theorem on the cohomology of the ``positive-depth part" of $Z$ in the tamely ramified case. This will be a crucial ingredient in proving the character formula of $\BR_{\BT}^{\BK}(\phi)$.

\subsection{A variation of Boyarchenko--Weinstein's calculation}
Recall that the natural map $\Gal(K/k) \to \Gal(k_f/k)$ admits a splitting. Thus $\Gal(K/k)$ is generated by $\t$ and a lift $F$ of Frobenius with $F^f = 1$.
%Note that the Galois group $\Gal(K/k)$ is generated by $\t$ and $F$ such that $F^f \in \<\t\> \subseteq \Gal(K/k)$. 
Note that $\pi^e = a \varpi$ for some root of unity $a \in k_f$. Let $\s: \BG_a \to \BG_a, x \mapsto x^q$ be the Frobenius map induced by $F$.

Let $C \subseteq \Psi$ be a $\tau$-orbit. In particular,  $V_C \neq 0$ and $C \neq -C$ by \Cref{V_C}. We identify ${}^E T_{r_C:r_C}$ with the $\ov\BF_q$-linear space $X_*(T) \otimes \ov\BF_q \pi^{e r_C}$. For $\a \in C$ we define an isomorphism  \[\th_\a: \BG_a \to (\BK^+)_C, \quad\ x \mapsto \sum_{i=0}^{|C|-1} \t^i(\a^\vee) \otimes \xi^{i e r_C} x \pi^{e r_C}. \]

Let $N$ be the minimal positive integer such that $F^N(C) = C$.
Suppose that $N$ is even and $F^{N/2}(C) = -C$. Let $\a \in C$ and $j \in \BZ$ such that \[F^{N/2} \circ \t^j (\a) = -\a.\] In particular, for $\g = \pm \a$ and $x \in \ov\BF_q$ we have \[F^{N/2}(\th_\g(x)) = \th_{-\g}(\eta \s^{N/2}(x)),\] where $\eta$ is the natural image of $F^{N/2}(\xi^j) (F(\pi)/\pi)^{e r_C}$ in $\ov\BF_q$. Let $\z \in \ov\BF_q$ such that $\z \s^{N/2} \z\i = \eta \s^{N/2} \in \ov\BF_q \rtimes \<\s\>$. Then we have $F^{iN}(\th_\a(x)) = \th_\a(\z \s^{iN} \z\i(x))$ for $i \in \BZ$.

As $F^f = 1 \in \Gal(K/k)$, for each $x \in \BF_{q^f}$ we have \[\th_\a(x) = F^f(\th_\a(x)) = \th_\a(\z \s^f \z\i (x)) ,\] which implies that $\z = \s^f(\z)$ and hence $\z \in \BF_{q^f}$. 

\begin{lemma} \label{image}
    Let $C$, $N$, $\a$ and $\z$ be as above. Then we have \[N_{K/k} \circ \a^\vee(K_{r_C}^\times) = (\sum_{i=0}^{N-1} F^i) \circ \th_\a(\z \BF_{q^N}).\]
\end{lemma}
\begin{proof}
    Let $x \in \BF_{q^f}$. We have \begin{align*}
        &\quad\ N_{K/k}(\a^\vee \otimes x \pi^{e r_C}) \\ &= (\sum_{i=0}^{f-1} F^i) (\sum_{l=0}^{e-1} \t^l) (\a^\vee \otimes x \pi^{e r_C}) \\ &= \frac{e}{|C|}(\sum_{i=0}^f F^i)(\th_\a(x)) \\ &= \frac{e}{|C|} (\sum_{i=0}^{N-1} F^i) (\sum_{l=0}^{\frac{f}{N} - 1}) (F^{l N}(\th_\a(x))) \\ &= \frac{e}{|C|} (\sum_{i=0}^{N-1} F^i) (\sum_{l=0}^{\frac{f}{N} - 1} \th_\a(\z \s^{l N} (\z\i x))) \\ &= \frac{e}{|C|} (\sum_{i=0}^{N-1} F^i)(\th_\a(\z \tr_{\BF_{q^f} / \BF_{q^N}}(\z\i x))).
    \end{align*} As $\z \in \BF_{q^f}$ and that $\frac{e}{|C|}$ is invertible in $\BF_q$, we deduce that \[N_{K/k}(\a^\vee \otimes \BF_{q^f} \pi^{e r_C}) = (\sum_{i=0}^{N-1} F^i) \circ \th_\a(\z \tr_{\BF_{q^f} / \BF_{q^N}}(\BF_{q^f})) = (\sum_{i=0}^{N-1} F^i) \circ \th_\a(\z \BF_{q^N}).\] Hence the statement follows.  
\end{proof}

Let $H$ be a connected commutative algebraic group over a finite field $\BF_Q$ of characteristic $p$. For any character $\chi: H(\BF_Q) \to \ov\BQ_\ell^\times$ we denote by $\CL_\chi$ the associated rank-one local system over $H$.

Let $\phi_+ = \phi |_{(\BT^{0+})^F}$.
\begin{proposition} \label{cohomology}
    Let $C$, $N$, $\a$ and $\z$ be as above. Then we have \begin{align*}
        \dim_{\ov\BQ_\ell} H_c^i(\BG_a, \iota^* \CL_{\phi_+}) = \begin{cases}
            q^{N/2} & \text{ if } i=1; \\ 0 & \text{ otherwise}.
        \end{cases}
    \end{align*} Here $\iota: \BG_a \to \BT_+$ is given by $x \mapsto \th_\a(\z x^{q^{N/2}+1})$.
\end{proposition}
\begin{proof}
    Let $\vartheta: \BG_a \to \BT^{0+}$ be given by $x \mapsto \th_\a(\z x)$. Note that $F^N \th_\a(x) = \th_\a(\z \s^N \z\i x)$ for $x \in \ov\BF_q$. Hence $\vartheta$ is an homomorphism of algebraic groups over $\BF_{q^N}$. Consider the following character \[\psi: \BF_{q^N} \overset \vartheta \to (\BT^{0+})^{F^n} \overset {N_F^{F^N}} \to (\BT^{0+})^F \overset {\phi_+} \to \ov\BQ_\ell^\times,\] where $N_F^{F^N}: \BT^{0+} \to \BT^{0+}$ is given by $x \mapsto x F(x) \cdots F^{N-1}(x)$. Then we have \[\CL_\psi \cong \vartheta^* \CL_{\phi_+ \circ N_F^{F^N}} \cong \vartheta^* \CL_{\phi_+},\] where the second isomorphism follows from that $\CL_{\phi_+ \circ N_F^{F^N}} \cong \CL_{\phi_+}$ by \cite[Lemma 5.8]{IvanovNie_24}.

    By \Cref{image} we set \[\Im \psi = \phi_+ \circ (\sum_{i=1}^{N-1} F^i)(\Im \vartheta) = \phi_{i_C-1} \circ (\sum_{i=1}^{N-1} F^i)(\Im \vartheta) = \phi_{i_C-1} \circ N_{K/k} \circ \a^\vee(K_{r_C}^\times) \neq \{1\},\] where the last inequality follows from that $\phi_{i_C-1}$ is $(G^{i_C-1}, G^{i_C})$-generic. Thus $\psi$ is non-trivial. On the other hand, for $x \in \BF_{q^{N/2}}$ we have \begin{align*}
        N_F^{F^N} \vartheta(x) &= N_F^{F^N} \th_\a(\z x) \\ &= (\sum_{i=0}^{N/2-1} F^i)(\th_\a(\z x) + F^{N/2} \th_\a(\z x)) \\ &= (\sum_{i=0}^{N/2-1} F^i)(\th_\a(\z x) + \th_{-\a}(\z \s^{N/2}(x))) \\ &=(\sum_{i=0}^{N/2-1}F^i)(\th_\a(\z x) + \th_{-\a}(\z x)) \\ &= 0.
    \end{align*} In particular, $\psi$ is trivial over $\BF_{q^{N/2}}$. Let $f: \BG_a \to \BG_a$ be given by $x \mapsto x^{q^{N/2}+1}$. We have $\iota = \vartheta \circ f$ and hence \[\iota^* \CL_{\phi_+} \cong f^* \vartheta^* \CL_{\phi_+} \cong f^* \CL_\psi.\] Then the statement follows from \cite[Proposition 6.6.1]{BoyarchenkoW_16}.   
\end{proof}

\subsection{The concentration result}
Let $s \in \widetilde\BK$. We denote by $\Ad(s): \widetilde \BK  \to  \widetilde \BK $ the adjoint action induced by $x \mapsto s x s\i$. Let $H$ be the quotient group of two closed subgroups of $\widetilde\BK$ preserved by $\Ad(s)$. We denote by $H^s$ for the group of $\Ad(s)$-fixed points in $H$, and write $H^{s, \circ}$ for the identity component of $H^s$.

We write $\bar T^F$ for the natural image of $T^F$ in $\widetilde\BK$. Let $t \in \bar T^F$. We set $\Phi^t = \{\g \in \Phi; \g(\hat t) = 1 \in \pi \CO_E\}$, where $\hat t \in T^F$ is any lift of $t$. Define \[d(\phi, t) = |(\Psi \cap \Phi^t) / \Gal(E / k)|,\] where $\Psi \subseteq \Phi$ is as in \S\ref{subsec:lagrangian}. Let $\k_t$ denote the Weil--Heisenberg representation of $\widetilde \BK^t$ associated to generic datum $\L = (G^i, \phi_i)_{i=0}^d$ in the sense of \cite[\S 4]{Yu_01}, see also \cite[\S 2]{Fin_21b}. We write $d(\phi, t) = d(\phi)$ and $\k_t = \k$ if $t = 1$. 
 
\begin{proposition} \label{concentration}
    Let $t \in \BT^F$. Then there exists an integer $n_{\phi, t}$ such that \[ H_c^i(Z \cap \BH^t, \ov\BQ_\ell)[\phi_+] \neq 0 \Longleftrightarrow i = n_{\phi, t}.\] Moreover, $\dim H_c^{n_{\phi, t}}(Z \cap \BH^t, \ov\BQ_\ell)[\phi_+] = |(V^t)^F|^{\frac{1}{2}}$ and $n_{\phi, t} \equiv d(\phi, t) \mod 2$.
\end{proposition}
\begin{proof}
    It follows the same way as \cite[Theorem 4.1]{LN_25} by using \Cref{cohomology}.
\end{proof}

\section{A character formula for $\hat \CR_{\BT}^{\BK}(\phi)$} \label{sec:character}

Keep the notation in previous sections. The main purpose of this section is to established a character formula for $\CR_{\BT}^{\BK}(\phi)$ and $\hat \CR_{\BT}^{\BK}(\phi)$ by following approaches in \cite{DeligneL_76} and \cite{CO_25}.

First we need two auxiliary lemmas.
\begin{lemma} \label{prime-to-p}
    For $t \in \bar T^F$ the conjugation action $\Ad(t)$ on $\BK$ is semisimple and has finite order prime to the characteristic of $\BF_q$.
\end{lemma}
\begin{proof}
    The statement follows from that $\BK$ is generated by $\BT$ and one-dimensional additive subgroups over $\ov\BF_q$ which are preserved by $\Ad(s)$. 
\end{proof}

\begin{lemma} \label{fix}
    Let $s \in \BK^F \bar T^F$ and $x \in \bar T^F$. The map $(h, y) \mapsto h y$ induces a bijection \[\th: \{h \in \BK^F; s h x =h\} \times_{(\BK^{x, \circ})^F} (Z \cap \BK^{x, \circ}) \cong \{g \in Z; s g x = g\}.\] 
\end{lemma}
\begin{proof}
    First note that map $\th$ is well defined. If $\th(h, y) = \th(h', y')$, then ${h'}\i h = y' y\i \in (\BK^{x, \circ})^F$ and hence $\th$ is injective. It remains to show $\th$ is surjective.

    Let $g \in Z$ such that $s g x = g$. Then $g\i F(g) \in F \BI$ and \[g\i F(g) = x\i (sg)\i F(sg) F(x) = s\i g\i F(g) s,\] which implies that $g\i F(g) \in (F\BI)^x = F \BI^x$.
    
    By \Cref{prime-to-p}, the conjugation action of $x \in \bar T$ on $\BK$ is semisimple, which implies that $\BI^x$ is connected (this follows from  \cite[Proposition 9.3]{Borel_91} by noting that $x$ normalizes the unipotent group $\BI$) 
    and hence $F\BI^x \subseteq \BK^{x, \circ}$. By Lang's theorem, there exists $y \in \BK^{x, \circ}$ such that $y\i F(y) = g\i F(g) \in F\BI$. Let $h = g y\i$. Then we have  $h \in \BK^F$, $g = h y$ and $s h x = s g y\i x = s g x y\i = g y\i = h$ as desired. Hence $\th$ is surjective and the proof is finished. 
\end{proof}

\begin{lemma} \label{J-dec}
    Let $g \in \BK^F \bar T^F$. Then there exist $z \in \bar T^F$, $u \in \BK^F$ and $s \in \BK^F \bar T^F$ such that $g = s u = u s$, $u \in \BK$ is unipotent and $(s z\i) \rtimes \Ad(z) \in \BK^F \rtimes \<\Ad(z)\>$ has finite order prime to $p$.

    Moreover, if $g \in \BK^F$ we may take $z = 1$ and $g = s u= u s$ is the Jordan decomposition in $\BK$.
\end{lemma}
\begin{proof}
    Assume $g = g' z$ for some $g \in \BK^F$ and $z \in \bar T^F$. By \Cref{prime-to-p}, the semi-direct product $\BK \rtimes \<\Ad(z)\>$ is a linear algebraic group. Let $x = g' \rtimes \Ad(z) \in \BK \rtimes \<\Ad(z)\>$ and let $x = x_u x_s$ be the usual Jordan decomposition with $x_u$ unipotent and $x_s$ semisimple. As $z \in \bar T^F$, the Frobenius automorphism of $\BK$ extends trivially to an automorphism of $\BK \rtimes \<\Ad(z)\>$. By the uniqueness of the Jordan decomposition, we have $x_u, x_s \in \BK^F \rtimes \<\Ad(z)\>$.
    %By uniqueness of Jordan decomposition, we have $x_u \in \BK_{\phi, r}^F$ and $x_s \in \BK_{\phi, r}^F \rtimes \<\Ad(z)\>$. 
    As the order of $x_u$ is a power of $p$, it follows from \Cref{prime-to-p} that the image of $x_u$ under the natural projection $\BK \rtimes \<\Ad(z)\> \to \<\Ad(z)\>$ is trivial, which means $x_u \in \BK^F$. Let $u = x_u$ and $s = \d z$, where $\d \in \BK^F$ such that $x_s = \d \rtimes \Ad(z)$. Note that the order of $x_s$ is prime to $p$. Hence the elements $z$, $u$ and $s$ satisfy our requirements.
\end{proof}

Let $\BT_\red$ be the reductive part of $\BT$. We have the Jordan decomposition $\BT^F \cong \BT_\red^F \times (\BT^{0+})^F$. Recall that $\phi_+ = \phi |_{(\BT^{0+})^F}$, and recall the  virtual $\bar T^F \BK^F$-module $\hat\CR_{\BT}^{\BK}(\phi) = H_c^*(Z, \ov\BQ_\ell)[\phi] \otimes \phi$ defined in \S\ref{subsec:extension}.

Similarly as in \cite{DeligneL_76, Chen, CO_25}, we have the following character formula.
\begin{proposition} \label{Jordan}
    Let notation be as in \Cref{J-dec}. Assume that $p \nmid |W_{\BL}(\BT)^F|$ if $g \in \BK^F \bar T^F \sm \BK^F$. Then \[\hat\CR_{\BT}^{\BK}(\phi)(g) = \frac{1}{|\BT_\red^F|} \sum_{h \in \BK^F / (\BK^{s, \circ})^F,\ {}^h s \in \bar T} \phi({}^h s) \tr({}^h u; H_c^*(Z \cap \BK^{{}^h s, \circ}, \ov\BQ_\ell)[\phi_+]).\]
\end{proposition}
\begin{proof}
    By definition we have \begin{align*} 
        &\quad\ \hat\CR_{\BT}^{\BK}(\phi)(g) \\ &=\phi(z) \tr((g z\i)\rtimes \Ad(z); H_c^*(Z, \ov\BQ_\ell)[\phi]) \\ &=\phi(z)\frac{1}{|\BT^F|}\sum_{t \in \BT^F} \phi(t\i) \tr((g z\i)\rtimes \Ad(z), t); H_c^*(Z, \ov\BQ_\ell)) \\ &=\phi(z)\frac{1}{|\BT^F|} \sum_{t \in \BT^F} \phi(t)\i \tr((u, t''); H_c^*(Z^{((sz\i)\rtimes \Ad(z), t')}, \ov\BQ_\ell)), \end{align*} where $t = t' t''$ is the Jordan decomposition with $t' \in \BT_\red$ and $t'' \in (\BT^{0+})^F$, and \[Z^{((sz\i)\rtimes \Ad(z), t')} = \{g \in Z; sz\i (z g z\i) t' = s g z\i t' = g\}\] is the set of points in $Z$ fixed by the action of $((sz\i)\rtimes \Ad(z), t')$. By \Cref{fix}, \[Z^{((sz\i)\rtimes \Ad(z), t')} = \bigsqcup_{h \in \BK^F / (\BK^{s, \circ})^F, {}^h s = z {t'}\i} h\i (Z \cap \BK^{{}^h s, \circ}).\] If $g \in \BK^F$ and $g = s u$ is the Jordan decomposition, then ${}^h u \in \BK^{{}^h s, \circ}$ since $\BK$ is connected. Otherwise, as $p \nmid |W_{\BL}(\BT)^F|$ (by assumption) and the order of $u$ is a power of $p$, we still have ${}^h u \in \BK^{{}^h s, \circ}$. Thus, \[\tr((u, t''); H_c^*(Z^{((sz\i)\rtimes \Ad(z), t')}, \ov\BQ_\ell)) = \sum_{h \in \BK^F / (\BK^{s, \circ})^F, {}^h s = z {t'}\i} \tr(({}^hu, t''); H_c^*(Z \cap \BK^{{}^h s, \circ}, \ov\BQ_\ell)).\] Inserting this equality we have
        \begin{align*} &\quad\ \hat\CR_{\BT}^{\BK}(\phi)(g) \\  &= \phi(z)\frac{1}{|\BT^F|}\sum_{t' \in \BT_\red^F} \sum_{t'' \in \BT_\unip^F} \phi(t' t'')\i \sum_{h \in \BK^F / (\BK^{s, \circ})^F, {}^h s = z {t'}\i} \tr(({}^h u, t''); H_c^*(Z \cap \BK^{{}^h s, \circ}, \ov\BQ_\ell)) \\ &= \phi(z)\frac{1}{|\BT^F|} \sum_{t'' \in (\BT_\unip)^F} \phi(t'')\i \phi(z\i {}^h s) \sum_{h \in \BK^F / (\BK^{s, \circ})^F, {}^h s \in \bar T} \tr(({}^h u, t''); H_c^*(Z \cap \BK^{{}^h s, \circ}, \ov\BQ_\ell)) \\ &=\frac{1}{|\BT_\red^F|} \sum_{h \in \BK^F / (\BK^{s, \circ})^F, {}^h s \in \bar T} \phi({}^h s) \tr({}^h u; H_c^*(Z \cap \BK^{{}^h s, \circ}, \ov\BQ_\ell)[\phi_+]).\end{align*} The proof is finished.
\end{proof}

Let $\Phi_{L_0}$ be the set of roots of $T_0$ in $L_0$. We say $s \in \BT_\red^F$ is regular for $L_0$ if $\a(s) \neq 1$ for all $\a \in \Phi_{L_0}$.
\begin{corollary} \label{CR-vreg}
    Let $s \in \BT_\red^F$ be regular for $L_0$. Then \[\CR_{\BT}^{\BK}(\phi)(s) = (-1)^{d(\phi, s)} |(V^s)^F|^{\frac{1}{2}} \sum_{w \in W_{\BL}(\BT)^F} \phi({}^w s).\]
\end{corollary}
\begin{proof}
    As $s \in \BT_\red^F$ is regular for $L_0$, we have $\BK^{s, \circ} = \BT \BH^s$ and \[\{h \in \BK / \BK^{s, \circ}; {}^h s \in \bar T\} \cong  \{h \in \BL / \BL^{s, \circ}; {}^h s \in \bar T\} \cong W_{\BL}(\BT)^F.\] Applying \Cref{Jordan} we have \begin{align*} \CR_{\BT}^{\BK}(\phi)(s) &= \frac{1}{|\BT_\red^F|} \sum_{h \in \BK^F / (\BK^{s, \circ})^F,\ {}^h s \in \BT} \phi({}^h s) \tr(1; H_c^*(Z \cap \BK^{{}^h s, \circ}, \ov\BQ_\ell)[\phi_+]) \\ &= \frac{1}{|\BT_\red^F|} \sum_{w \in W_{\BL}(\BT)^F} \phi({}^w s) \tr(1; H_c^*(Z \cap \BK^{{}^w s, \circ}, \ov\BQ_\ell)[\phi_+]) \\ &=\frac{1}{|\BT_\red^F|} \sum_{w \in W_{\BL}(\BT)^F} \phi({}^w s) \dim H_c^*(\bigsqcup_{\g \in \BT_\red^F} \g (Z \cap \BH^{{}^w s}), \ov\BQ_\ell)[\phi_+] \\ &= \sum_{w \in W_{\BL}(\BT)^F} \phi({}^w s) \dim H_c^*(Z \cap \BH^{{}^w s}, \ov\BQ_\ell)[\phi_+] \\ &= \sum_{w \in W_{\BL}(\BT)^F} \phi({}^w s) (-1)^{d(\phi, {}^w s)} |(V^{{}^w s})^F|^{\frac{1}{2}} \\ &=(-1)^{d(\phi, s)} |(V^s)^F|^{\frac{1}{2}} \sum_{w \in W_{\BL}(\BT)^F } \phi({}^w s), \end{align*} where the third equality follows from that \[Z \cap \BK^{{}^w s, \circ} = \bigsqcup_{\g \in \BT_\red^F} \g (Z \cap \BH^{{}^w s}),\] the fifth follows from \Cref{concentration}, the last one follows from that $d(\phi, s) = d(\phi, {}^w s)$ and $|(V^s)^F| = |(V^{{}^w s})^F|$. The proof is finished.
\end{proof}

\section{The module $\hat R_{\BT}^{\BK}(\phi^\dag)$} \label{sec:R-mod}
In this section, we introduce and study another $\bar T^F \BK^F$-module $\hat R_{\BT}^{\BK}(\phi^\dag)$, which serves as a counterpart of $\hat \CR_{\BT}^{\BK}(\phi)$.

Consider the following $\bar T^F \BK^F$-module \[\hat R_{\BT}^{\BK}(\phi) = \k|_{\bar T^F \BK^F} \otimes \hat R_{T_0}^{L_0}(\phi_{-1}),\] where $\hat R_{T_0}^{L_0}(\phi_{-1})$ is the extended classical Deligne--Lusztig representation of $\bar T^F \BK^F$ constructed in \cite[\S 3.4.4]{Kaletha_19}. 

Following \cite{ChanO_21}, let \[\varepsilon_{\phi_+} = \prod_{i=1}^d \e_{\sharp, \bx}^{G^i/G^{i-1}},\] where $\e_{\sharp, \bx}^{G^i/G^{i-1}}$ denotes the quadratic character of $\bar T^F$ in \cite[Definition 3.1]{FKS} which factors through $T^F / T_{\bx, 0+}^F$. We put $\phi^{\dag} = \varepsilon_{\phi_+} \cdot \phi$. Then $(\phi^{\dag})^{\dag} = \phi$. 

Let $\phi_{\ge 0} = \prod_{i=0}^d \phi_i$. The following result is essentially proved in \cite[Proposition 7.4]{CO_25}.
\begin{proposition}\label{trace-kappa}
    Let notation be as in \Cref{J-dec} and assume that $s \in \bar T^F$. Then \[\k(g) = (-1)^{d(\phi) - d(\phi, s)}  \k_s (u) \varepsilon_{\phi_+}(s) \phi_{\ge 0}(s).\] In particular, \[\k(s) = (-1)^{d(\phi) - d(\phi, s)}  |(V^s)^F| \varepsilon_{\phi_+}(s) \phi_{\ge 0}(s).\] 
\end{proposition}

\begin{lemma} \label{stab-quotient}
    Let $s \in \BK \bar T^F$ such that the action $\Ad(s)$ on $\BK$ has order prime to $p$. Then the natural quotient map $\pi: \BK \to L_0$ induces a bijection \[\a: \{h \in \BK / \BK^{s, \circ}; {}^h s \in \BT \bar T^F\} \cong \{x \in L_0 / L_0^{s, \circ}; {}^x \pi(s) \in T^F / T_{\bx, 0+}^F\}.\]
\end{lemma}
\begin{proof}
    %(1) For $t \in\BR_{\ge 0}$ we denote by $\BK_t$, $\BT_t$ and $\BH_t$ the natural images of $\BK$, $\BT$ and $\BH$ in $\BG_t$. We argue by induction on $t$ that the image $s_t$ of $s$ in $\BK_t$ is conjugate to an element of $\BT_t$. If $t = 0$, the statement follows by noticing that $s_0 \in L_0$ is semisimple and that $T_0$ is a maximal torus of the reductive group $L_0$. Assume that the statement is true for $t-$ with $0 \le t-$. By induction hypothesis we may assume $s_{t-} \in \BT_{s-}$. By conjugating $s_t$ with a suitable element in $\BH_{t:t} := \ker(\BH_t \to \BH_{t-})$, we may assume that $s_t = xy = yx$, where $x \in \BT_t$ is semisimple and $y \in \BH_{t:t}$. Since $s_t$ is semisimple and $y$ is unipotent, we deduce that $y = 1 \in \BH_{t:t}$ and hence the induction procedure is finished. Therefore, (1) is proved.

    First we show that $\a$ is injective. We may assume $s \in \BT\bar T^F$. Then the natural projection $\BL^{s, \circ} \to L_0^{s, \circ}$ is surjective. Let $h \in \BK$ such that ${}^h s \in \BT \bar T^F$ and $\pi(h) \in L_0^{s, \circ}$. We need to show that $h \in \BK^{s, \circ}$. Let $h' \in \BL^{s, \circ}$ be a lift of $\pi(h)$ under this projection. By replacing $h$ with $h {h'}\i$, we can assume that $h \in \ker \pi$, that is, $h \in \BH$. Then the condition ${}^h s \in \bar T$ implies that $h \in \BH^s = \BH^{s, \circ}$. Hence $\a$ is injective.

    Now we show that $\a$ is surjective. We may assume $s \in \BH \bar T^F$ and $x = 1 \in L_0$. Write $s = \d t$ with $\d \in \BH$ and $t \in \BT\bar T^F$. As the action of $\Ad(t)$ on $\BH$ is a semisimple and $[\BH, \BH]$ lies in the center of $\BH$, by replacing $s$ with a suitable $\BH$-conjugate, we may assume that $s = \d t = t \d$. Assume that $\d \notin \BT$. Let $N$ be the order of $\Ad(s)$. We have $s^N = \d^N t^N$. As $p \nmid N$ by assumption, $\d^N \notin \BT$. Hence the action of $\Ad(\d^N t^N) = \Ad(s^N) = \Ad(s)^N$ on $\BH$ is nontrivial, which is a contradiction. So $\d \in \BT$ and $\a$ is surjective as desired.
\end{proof}

We have the following character formula.
\begin{proposition} \label{K-trace}
     Let notation be as in \Cref{J-dec}. Then \[\hat R_{\BT}^{\BK}(\phi^\dag)(g) = (-1)^{d(\phi)}  \sum_{h \in \BK^F / (\BK^{s, \circ})^F, {}^h s \in \bar T^F} (-1)^{d(\phi, {}^h s)} Q_{{}^h s, \phi_+}({}^h u)  \phi({}^h s).\]
\end{proposition}
\begin{proof}
     By definition we have \begin{align*}
         &\quad\ \hat R_{\BT}^{\BK}(\phi^\dag)(g) \\ &= \k(g) \hat R_{T_0}^{L_0}(\phi_{-1}^\dag)(g) \\ &=  \k(g) \sum_{x \in L_0^F / (L_0^{s, \circ})^F, {}^x \pi(s) \in T^F / T_{\bx, 0+}^F}  \phi_{-1}^\dag({}^x s)  Q_{L_0^{s, \circ}, T_0}({}^x u)\\ &= \sum_{h \in \BK^F / (\BK^{s, \circ})^F, {}^h s \in \bar T^F} \k({}^h g)  \phi_{-1}^\dag({}^h s) Q_{L_0^{{}^h s, \circ}, T_0}({}^h u) \\ &=  \sum_{h \in \BK^F / (\BK^{s, \circ})^F, {}^h s \in \bar T^F} (-1)^{d(\phi) - d(\phi, {}^h s)}  \k_{{}^h s}({}^h u) \phi_{\ge 0}({}^h s) \varepsilon_{\phi_+}({}^h s) \phi_{-1}^\dag({}^h s) Q_{L_0^{{}^h s, \circ}, T_0}({}^h u) \\ &=  (-1)^{d(\phi)} \sum_{h \in \BK^F / (\BK^{s, \circ})^F, {}^h s \in \bar T^F} (-1)^{d(\phi, {}^h s)} \k_{{}^h s}({}^h u) Q_{L_0^{{}^h s, \circ}, T_0}({}^h u) \phi({}^h s) \\  &=  (-1)^{d(\phi)} \sum_{h \in \BK^F / (\BK^{s, \circ})^F, {}^h s \in \bar T^F} (-1)^{d(\phi, {}^h s)} Q_{{}^h s, \phi_+}({}^h u) \phi({}^h s),
     \end{align*} where the third equality follows from \Cref{stab-quotient}, and the fourth follows from \Cref{trace-kappa}.
\end{proof}

\section{Green functions} \label{sec:Green}
In this section, we introduce Green functions associated to the $\BK^F$-modules $\CR_{\BT}^{\BK}(\phi)$ and $R_{\BT}^{\BK}(\phi)$, and prove \Cref{main-2} based on a strategy of \cite{CO_25}.

For $t \in \bar T^F$ we consider the following virtual $(\BK^{t, \circ})^F$-modules \[\CR_{\BT}^{\BK^{t, \circ}}(\phi) = H_c^*(Z \cap \BK^{t, \circ}, \ov\BQ_\ell)[\phi], \quad\  R_{\BT}^{\BK^{t, \circ}}(\phi) = \k_t|_{(\BK^{t, \circ})^F} \otimes R_{T_0}^{L_0^{t, \circ}}(\phi_{-1}),\] where $R_{T_0}^{L_0^{t, \circ}}(\phi_{-1})$ denotes the classical Deligne--Lusztig representation of $(L_0^{t, \circ})^F$ associated to the pair $(T_0, \phi)$. Following \cite[\S4]{DeligneL_76} and \cite[\S6-7]{CO_25}, we consider for $t \in \bar T^F$ the following deep level Green functions \[\CQ_{t, \phi_+} = \CR_{\BT}^{\BK^{t, \circ}}(\phi)|_{(\BK^{t, \circ}_\unip)^F} \text{ and } Q_{t, \phi_+} = R_{\BT}^{\BK^{t, \circ}}(\phi)|_{(\BK^{t, \circ}_\unip)^F}.\] Here $\BK^{t, \circ}_\unip$ denotes the set of unipotent elements in $\BK^{t, \circ}$. Write $\CQ_{t, \phi_+} = \CQ_{\phi_+}$ and $Q_{t, \phi_+} = Q_{\phi_+}$ if $t = 1$. By definition, \[Q_{t, \phi_+} = \k_t|_{(\BK^{t, \circ}_\unip)^F} \cdot Q_{L_0^{t, \circ}, T_0}.\] Here $Q_{L_0^{t, \circ}, T_0}$ denotes the classical Green function attached to the pair $(L_0^{t, \circ}, T_0)$ in the sense of \cite{DeligneL_76}.

\begin{proposition} \label{independent}
Both $\CQ_{\phi_+}$ and $Q_{\phi_+}$ only depend on the positive-depth part $\phi_+$ of $\phi$.\end{proposition} 
\begin{proof} Let $u \in \BK^F_\unip$. By \Cref{Jordan} we have \[\CQ_{\phi_+}(u) = \CR_{\BT}^{\BK}(\phi)(u) = \frac{1}{|\BT_\red^F|} \tr(u; H_c^*(Z, \ov\BQ_\ell)[\phi_+]),\] which only depends on $\phi_+$ as desired. The other statement follows from that the classical Green function $R_{T_0}^{L_0}(\phi_{-1})|_{(L_0)^F_\unip}$ is independent of $\phi_{-1}$ by \cite[Theorem 4.2]{DeligneL_76}. \end{proof} 

%Let $s \in \bar T^F$. let $T \subseteq M = G(s) \subseteq G$ be the subgroup such that the root system of $T$ in $M$ is $\Phi_s$. As in \Cref{subsec:subgroup}, we have an associated Howe factorization $(\phi_{-1}, \L_M)$ of $\phi$ for $M$, the subgroup $\CK^M_\phi$ to $(\phi_{-1}, \L_M)$ and a virtual $(\CK^M_\phi)^F$-module $\CR_{T_r}^{\CK^M_r}(\phi)$. On the other hand, we have another virtual $(\CK^M_\phi)^F$-module \[ R_{T_r}^{\CK^M_r}(\phi) = \k_s \otimes R_{T_0}^{L_0^{s, \circ}}(\phi_{-1}),\] where $\k_s := \k_{\L_M}$ is the Weil--Heisenberg representation attached to the generic datum $\L_M$ for $M$. Similarly, we put \[\CQ_{s, \phi_+} = \CR_{T_r}^{\CK^M_r}(\phi)|_{(\CK^M_{\phi, r})^F_\unip}, \quad \ Q_{s, \phi_+} = R_{T_r}^{\CK^{G(s)}_r}(\phi)|_{(\CK^M_{\phi, r})^F_\unip}.\] 

%\begin{lemma} \label{G(s)}Let $s \in \bar T^F$. Then $\CR_{T_r}^{\CK^{G(s)}_r}(\phi) = H_c^*(Z \cap \BK^{s, \circ}, \ov\BQ_\ell)[\phi]$. \end{lemma}
%\begin{proof} Let $M = G(s)$. There is a natural isomorphism \[(Z \cap M_r) / (\CE_{\phi, r} \cap M_r) \cong Z \cap \BK^{s, \circ}.\] By \Cref{Levi} and the observation that $\CE_{\phi, r} \cap M_r$ is an affine space, it follows that \[\CR_{T_r}^{\CK^{G(s)}_r}(\phi) = H_c^*(Z \cap M_r, \ov\BQ_\ell)[\phi] = H_c^*(Z \cap \BK^{s, \circ}, \ov\BQ_\ell)[\phi]\] as desired. \end{proof}

\begin{theorem} \label{CK-trace}
      Let notation and assumption be as in \Cref{Jordan}. Then \[\hat \CR_{\BT}^{\BK}(\phi)(g) = \sum_{h \in \BK^F / (\BK^{s, \circ})^F,\ {}^h s \in \bar T^F} \phi({}^h s) \CQ_{{}^h s, \phi_+}({}^h u).\]
\end{theorem}
\begin{proof}
    Applying \Cref{Jordan} we have \begin{align*}
        \hat\CR_{\BT}^{\BK}(\phi)(g) &= \sum_{h \in \BK^F / (\BK^{s, \circ})^F,\ {}^h s \in \BT} \phi({}^h s) \frac{1}{|\BT_\red^F|} \tr({}^h u; H_c^*(Z \cap \BK^{{}^h s, \circ}, \ov\BQ_\ell)[\phi_+]) \\ &= \sum_{h \in \BK^F / (\BK^{s, \circ})^F,\ {}^h s \in \BT} \phi({}^h s) \tr({}^h u; H_c^*(Z \cap \BK^{{}^h s, \circ}, \ov\BQ_\ell)[\phi]) \\ &= \sum_{h \in \BK^F / (\BK^{s, \circ})^F,\ {}^h s \in \BT} \phi({}^h s) \CQ_{{}^h s, \phi_+}({}^h u).
    \end{align*} The proof is finished.   
\end{proof}

We say $\phi$ is \emph{regular} for $L_0$ if the stabilizer of $\phi|_{\BT^F}$ in $W_{\BL}(\BT)^F \cong W_{\BL}(\BT)^F$ is trivial. 
\begin{proposition} \label{tensor}
    If $\phi$ is regular for $L_0$, then there exists an irreducible $L_0^F$-module $\rho$ such that \[\pm \CR_{\BT}^{\BK}(\phi) \cong \k \otimes \rho.\] In particular, $\pm \CR_\BT^\BK(\phi)$ is an irreducible $\BK^F$-module.
\end{proposition}
\begin{proof}
    By \Cref{product}, $\pm \CR_{\BT}^{\BK}(\phi)$ is an irreducible $\BK^F$-module. By construction, $(\BK^+)^F \cong (\BT^{0+})^F$ lies in the center of $\BK$ and hence acts on $\pm \CR_{\BT}^{\BK}(\phi)$ via the character $\phi$. Hence the restriction $\pm\CR_{\BT}^{\BK}(\phi)|_{\BH^F}$ is a sum of the Heisenberg representations isomorphic to $\k|_{\BH^F}$. In particular, \[\hom_{\BK^F}(\ind_{\BH^F}^{\BK^F} \k|_{\BH^F}, \pm\CR_{\BT}^{\BK}(\phi)) = \hom_{\BH^F}(\k|_{\BH^F}, \pm\CR_{\BT}^{\BK}(\phi)|_{\BH^F}) \neq 0.\] By the projection formula, $\ind_{\BH^F}^{\BK^F} \k|_{\BH^F} \cong \k \otimes \ind_{\BH^F}^{\BK^F} 1$ and hence is a direct sum of $\BK^F$-modules of the form $\k \otimes \rho$, where $\rho$ ranges over irreducible summands (with multiplicities) of $\ind_{\BH^F}^{\BK^F} 1$. The statement then follows form that $\pm\CR_{\BT}^{\BK}(\phi)$ is an irreducible $\BK^F$-module.
\end{proof}

\begin{proposition} \label{regular-case}
    Assume $q$ is sufficiently large. Then $\CQ_{t, \phi_+} = (-1)^{d(\phi, t)} Q_{t, \phi_+}$ for $t \in \bar T^F$.
\end{proposition}
\begin{proof}
    Without loss of generality we may assume $t = 1$. As $q$ is sufficiently large, there exists a regular character $\psi: T^F \to \ov\BQ_\ell^\times$ such that $\psi_+ = \phi_+$. In view of \Cref{independent}, by replacing $\phi$ with $\psi$ we may assume further that $\phi$ is regular.
    
    By \Cref{tensor}, there exists an irreducible $L_0^F$-module $\rho$ such that $\pm \CR_{\BT}^{\BK}(\phi) \cong \k \otimes \rho$. Let $s \in \BT^F$ which is regular for $L_0$. By \Cref{CR-vreg} and \Cref{K-trace} we have \begin{align*} \CR_{\BT}^{\BK}(\phi)(s) &= (-1)^{d(\phi, s)} |(V^s)^F|^{\frac{1}{2}} \sum_{w \in W_{\BL}(\BT)^F} \phi({}^w s); \\ (\k \otimes \rho)(s) &= (-1)^{d(\phi) - d(\phi, s)} |(V^s)^F|^{\frac{1}{2}} \varepsilon_{\phi_+}(s) \phi_{\ge 0}(s) \rho(s). \end{align*} Since $\phi = \phi_{-1} \phi_{\ge 0}$ and $\pm\CR_{\BT}^{\BK}(\phi) \cong \k \otimes \rho$ we deduce that \[\pm \rho(s) = \sum_{w \in W_{\BL}(\BT)^F} \phi_{-1}^\dag({}^w s) = \sum_{w \in W_{L_0}(T_0)^F} \phi_{-1}^\dag({}^w s) = R_{T_0}^{L_0}(\phi_{-1}^\dag)(s).\] On the other hand, as $\phi$ is regular, so are $\phi_{-1}$ and $\phi_{-1}^\dag$. Hence $\pm R_{T_0}^{L_0}(\phi_{-1}^\dag)$ is irreducible. As $q$ is sufficiently large, it follows from \cite[Theorem 3.2]{CO_25} that $\rho = \pm R_{T_0}^{L_0}(\phi_{-1}^\dag)$ and hence \[\pm \CR_{\BT}^{\BK}(\phi) = \k \otimes R_{T_0}^{L_0}(\phi_{-1}^\dag) = R_{\BT}^{\BK}(\phi^\dag).\] By comparing traces on both sides, we deduce that $\CR_{\BT}^{\BK}(\phi) = (-1)^{d(\phi)} R_{\BT}^{\BK}(\phi^\dag)$ as desired.
\end{proof}

\begin{theorem}\label{thm:final_comparison}
    Assume $q$ is sufficiently large. Then we have \[\CR_{\BT}^{\BK}(\phi) = (-1)^{d(\phi)} R_{\BT}^{\BK}(\phi^\dag).\] If, moreover, $p \nmid |W_{\BL}(\BT)^F|$, then $\hat \CR_{\BT}^{\BK}(\phi) = (-1)^{d(\phi)} \hat R_{\BT}^{\BK}(\phi^\dag)$.
\end{theorem}
\begin{proof}
    We only show the first statement. The second follows in the same way. Let notation be as in \Cref{Jordan}. Combining \Cref{K-trace}, \Cref{CK-trace} and \Cref{regular-case} we have \begin{align*}
        R_{\BT}^{\BK}(\phi^\dag)(g) &= (-1)^{d(\phi)}  \sum_{h \in \BK^F / (\BK^{s, \circ})^F, {}^h s \in \BT^F} (-1)^{d(\phi, {}^h s)} Q_{{}^h s, \phi_+}({}^h u) \phi({}^h s) \\ &= (-1)^{d(\phi)}  \sum_{h \in \BK^F / (\BK^{s, \circ})^F, {}^h s \in \BT} \CQ_{{}^h s, \phi_+}({}^h u) \phi({}^h s) \\ &= (-1)^{d(\phi)} \CR_{\BT}^{\BK}(\phi)(g).
    \end{align*} The proof is finished.    
\end{proof}

\subsection{Regular supercuspidal representations}
Following \cite[\S 3.7]{Kaletha_19}, we assume further that $(T, \phi)$ is a tame elliptic regular pair, and let $\pi_{(T, \phi)}$ denote Kaletha's regular supercuspidal representation associated to $(T, \phi)$. 
\begin{corollary}
    Assume $p \nmid |W_{\BL}(\BT)|$ and $q$ is sufficiently large. Then $(-1)^{d(\phi) + r_L - r_T} \hat \CR_{\BT}^{\BK}(\phi)$ is an irreducible $\CK_\phi^F T^F$-module and \[\pi_{(T, \phi)} \cong \text{c-}\ind_{\CK_\L^F T^F}^{G^F} (-1)^{d(\phi) + r_L - r_T} \hat \CR_{\BT}^{\BK}(\phi^\dag).\] Here $r_L$ and $r_T$ denote the split ranks of $L$ and $T$ respectively.
\end{corollary}
\begin{proof}
    Note that $ (-1)^{r_L - r_T} \hat R_{\BT}^{\BK}(\phi)$ is an irreducible $\CK_\L^F T^F$-module. By \cite[Lemma 3.10]{ChanO_21} we have \[\pi_{(T, \phi)} \cong \text{c-}\ind_{\CK_\L^F T^F}^{G^F} (-1)^{r_L - r_T} \hat R_{\BT}^{\BK}(\phi).\] Note that $(\phi^\dag)^\dag = \phi$ and $d(\phi)=d(\phi^\dag)$. The statement then follows from \Cref{thm:final_comparison} that $\hat \CR_{\BT}^{\BK}(\phi^\dag) = (-1)^{d(\phi)} \hat R_{\BT}^{\BK}(\phi)$.
\end{proof}

\section{Supercuspidal representations} \label{sec:supercuspidal} 
In this section, we show that all the irreducible supercuspidal representations of $G^F$ can be realized by the cohomology of our varieties under mild assumptions on $p$ and $q$. 

First we recall that a Yu's datum is a triple $\Sigma = (\L, \bx, \widetilde \rho)$ such that
\begin{itemize}
    \item $\L = (G^i, \phi_i, r_i)_{i=0}^d$ is a generic datum as in \S\ref{sec:Howe};

    \item $Z(G^0) / Z(G)$ is anisotropic;

    \item $\bx \in \CB(G^0, k)$ is a vertex;

    \item $\widetilde \rho$ is a cuspidal representation of $(G^0_{[\bx]} / G_{\bx, 0+})^F$.
\end{itemize}

We fix a Yu's datum $\Sigma = (\L, \bx, \widetilde \rho)$ as above and put $L = G^0$ for simplicity. Let \[\CK := \CK_\L \subseteq \widetilde \CK_\L =: \widetilde \CK\] be the Yu-type subgroups associated to $\L$, see \S\ref{sec:Howe}. Let $\k$ be the Weil--Heisenberg representation of $\widetilde \CK^F$ in the sense of \cite{Yu_01}.
\begin{theorem} [\cite{Yu_01, Kim, Fin_21a}]
    Assume $p$ does not divide the order of the absolute Weyl group of $G$. Then the map \[\Sigma \mapsto \pi_\Sigma:= \text{c-}\ind_{\widetilde K^F}^{G^F} \k \otimes \widetilde \rho\] induces a bijection between equivalence classes of Yu's data and isomorphism classes irreducible supercuspidal representations of $G^F$.
\end{theorem}

By assumption, the restriction $\widetilde \rho |_{Z(G)^F L_{\bx, 0}^F}$ has an irreducible cuspidal summand $\rho$. By \cite{DeligneL_76}, there exist an elliptic torus $\mathsf T \subseteq L_0$ and a character $\th: \mathsf T^F \to \ov\BQ_\ell^\times$ such that $\rho |_{L_{\bx, 0}^F}$ appears in the inflation of $R_{\mathsf T}^{L_0}(\th)$. By \cite[Lemma 3.4.4]{Kaletha_19}, there exists a maximally unramified elliptic maximal torus $T$ of $L$ such that $\bx \in \CA(T, \brk)$ and $\mathsf T \cong T_0$. By inflation we view $\th$ as character of $T_{\bx, 0}^F$. Let $\o$ be the central character of $\widetilde \rho$. By construction, $\o$ and $\th$ coincide over $Z(G)^F \cap T_{\bx, 0}^F$.

Since $Z(L)/Z(G)$ is anisotropic, $T$ is also elliptic in $G$. The following lemma shows that the quotient $T^F / (Z(G)^F T_{\bx, 0}^F)$ is a finite group. 
\begin{lemma}\label{lm:quot_torus_finite}
For an elliptic torus $T \subseteq G$, the quotient $T^F/(Z(G)^F T_{{\bf x},0}^F)$ is finite. 
\end{lemma}

\begin{proof}
The Iwahori subgroup $T_{{\bf x},0}^F$ has finite index in the maximal bounded subgroup $T(k)_b$ of $T(k)$ by \cite[Lemma 2.5.14]{KalethaPrasad_book}. So it suffices to show that $T(k)/(Z(G)(k) T(k)_b)$ is finite. Let $T_{\der} = T \cap G_{\der}$. Then $T_{\der}$ is anisotropic and hence $T_{\rm der}(k)_b = T_{\rm der}(k)$ by \cite[Proposition 2.5.8]{KalethaPrasad_book}. Let $D = G/G_{\rm der}$. Taking Galois-fixed points of the short exact sequence $T_{\der} \hookrightarrow T \twoheadrightarrow D$ and using that $T_{\der}(k)_b \subseteq T(k)_b$ (by \cite[Proposition 2.5.9]{KalethaPrasad_book}), and that $H^1(k,T_{\der})$ is finite (by \cite[Chap.III, Theorem 4]{Serre_Galois_cohomology_engl}), we see that it suffices to show that the image of the composition $Z(G)(k) \to T(k) \to D(k)$ has finite index in $D(k)$. But this follows by taking Galois cohomology of the short exact sequence $Z(G) \cap G_\der \hookrightarrow Z(G) \twoheadrightarrow D$ from the finiteness of $H^1(k,Z(G) \cap G_\der)$.
\end{proof}
Let $\phi_{-1}$ be a character of $T^F$ which extends both $\o$ and $\th$. In particular, $\phi_{-1}$ is of depth $0$. Let \[\phi = \phi_{-1} \prod_{i=0}^d \phi_i |_{T^F}.\] Then $(\phi_{-1}, \L)$ is a Howe factorization of $\phi$. Let $Z$ be the deep level Deligne--Lusztig variety constructed in \S\ref{subsec:variety}. We extend the $\CK^F$-modules $H_c^i(Z, \ov\BQ_\ell)[\phi]$ to $Z(G)^F \CK^F$-modules on which $Z(G)^F$ acts via $\phi_{-1}$ or $\o$. 

\begin{theorem}
    Let notation be as above. Assume $q$ is sufficiently large. Then $\pi_\Sigma$ is a summand of $\text{c-}\ind_{Z(G)^F \CK^F}^{G^F} H_c^i(Z, \ov\BQ_\ell)[\phi^\dag]$ for some integer $i \in \BZ_{\ge 0}$.
\end{theorem}
\begin{proof}
    Let $H$ be a group and let $M$ be a virtual $H$-module over $\ov\BQ_\ell$. Suppose that $M = \sum_\chi c_\chi \chi$, where $c_\chi \in \BZ$ and $\chi$ runs over isomorphism classes of irreducible $H$-modules. We put $|M| = \sum_\chi |c_\chi| \chi$, which is a genuine $H$-module.

    We view the virtual $L_{\bx, 0}^F$-module $R_{\mathsf T}^{L_0}(\th)$ as a natural $Z(G)^F \CK^F$-module on which $\CH_\L^F $ acts trivially and $Z(G)^F$ acts via $\phi_{-1}$.  By construction we have \[\hom_{\widetilde \CK^F}(\ind_{Z(G)^F \CK^F}^{\widetilde\CK^F} \rho, \widetilde \rho) \cong \hom_{Z(G)^F \CK^F}(\rho, \widetilde \rho) \neq 0.\] As $\rho$ appears in $R_{\mathsf T}^{L_0}(\th)$, it follows that $\k\otimes \widetilde\rho$ appears in \[\k \otimes \ind_{Z(G)^F \CK^F}^{\widetilde\CK^F} |R_{\mathsf T}^{L_0}(\th)| \cong \ind_{Z(G)^F \CK^F}^{\widetilde\CK^F} \k \otimes |R_{\mathsf T}^{L_0}(\th)| \cong \ind_{Z(L)^F \CK^F}^{\widetilde\CK^F} |\k \otimes R_{\mathsf T}^{L_0}(\th)|.\] On the other hand, by \Cref{thm:final_comparison} we have $\pm \k \otimes R_{\mathsf T}^{L_0}(\th) = H_c^*(Z, \ov\BQ_\ell)[\phi^\dag]$. In particular, there exists $i \in \BZ_{\ge 0}$ such that $\k \otimes \widetilde\rho$ is a direct summand of $\ind_{Z(G)^F \CK^F}^{\widetilde\CK^F} H_c^i(Z, \ov\BQ_\ell)[\phi^\dag]$. Therefore, $\pi_\Sigma = \text{c-}\ind_{\widetilde\CK^F}^{G^F} \k \otimes \widetilde\rho$ is a direct summand of \[\text{c-}\ind_{\widetilde\CK^F}^{G^F} \ind_{Z(G)^F \CK^F}^{\widetilde\CK^F} H_c^i(Z, \ov\BQ_\ell)[\phi^\dag] \cong \text{c-}\ind_{Z(G)^F\CK^F}^{G^F} H_c^i(Z, \ov\BQ_\ell)[\phi^\dag].\] The proof is finished.  
\end{proof}

\bibliography{bib_ADLV}{}
\bibliographystyle{amsalpha}

\end{document}